\documentclass[11pt]{amsart}

\usepackage{amsmath}
\usepackage{amssymb}
\usepackage{graphicx}
\newtheorem{theorem}{Theorem}[section]

\theoremstyle{definition}
\newtheorem{definition}{Definition}[section]

\newtheorem{remark}[theorem]{Remark}
\newtheorem{assumption}{Assumption}[section]
\numberwithin{equation}{section}

\newcommand{\diam}{\operatorname{diam\,}}

\newcommand{\tr}{\operatorname{tr\,}}

\def\al{\alpha}
\def\be{\beta}
\def\gm{\gamma}

\def\dl{\delta}

\def\lm{\lambda}

\title[Stability in terms of two measures of  semiflows ]{Stability in terms of two measures of  semiflows in space  $\mathrm{conv}\,(\Bbb R^n)$}

\author[V.I. Slyn'ko]{Vitaliy Ivanovich Slyn'ko}

\address[V.I. Slyn'ko]{S.P. Timoshenko Institute of Mechanics of NAS of Ukraine, Ukraine}
\email{{\tt vitstab@ukr.net}}

\author[V.S. Denysenko]{ Viktor Sergeevich Denysenko}
\address[V.S. Denysenko]{"Bohdan Khmelnitsky"$\,$  Cherkassy National University, Ukraine}
\email{\tt den\underline{\;\;}vik@ukr.net}

\author[E.V. Ocheretnyuk]{ Eugen Volodymirovich Ocheretnyuk}
\address[E.V. Ocheretnyuk]{Cherkassy State Technological University, Ukraine}
\email{\tt ocheretnyukeugen@ukr.net}



\begin{document}

\begin{abstract}
The stability problem in terms of two   measures for semiflows in space $\mathrm{conv}\,(\Bbb R^n)$     was investigated. On the basis of comparison principle the obtained  result is used to study  the stability  criteria for a certain  semiflow in  space $\mathrm{conv}\,(\Bbb R^n)$. This semiflow, in particular, generalizes set differential equations and a set of attainability for linear control systems. The sufficient conditions of stability and practical stability of  semiflow in terms of two measures was established. As measures the Hausdorff metric is considered as well as a special measures constructed on the basis of  the certain  mixed volumes. A significant number of examples of studying the stability for specific semiflows was given to illustrate the effectiveness of proposed approach.
\end{abstract}

\maketitle

%%%%%%%%%%%%%%%%%%%%%%%%%%%%%%%%%%%%%%%%%%%%%%%%%%%%%%%%%%%%%%%%%%%%%%%%%%%%%%%%

\section{Introduction}
The study of set differential equations (SDEs) in a metric space  was   initiated in monograph  \cite{lak}.
The  basic  theory questions:  the existence of solution of initial value problem, global existence of solutions, convergence of successive approximations and others, as well as the direct Lyapunov method and comparison method
were  discussed  therein.  Tasks of stability of stationary solutions for SDEs based on Lyapunov functions were also studied in  \cite{devi}. The authors  of paper \cite{devi} consider  the differential equations of the form
\begin{equation}\label{1-1}
D_Hu=F(t,u),\quad u(t_0)=u_0\in\mathrm{conv}\,(\Bbb R^n),
\end{equation}
where $u\in\mathrm{conv}\,(\Bbb R^n)$, $F\in C(\Bbb R_+\times\mathrm{conv}\,(\Bbb R^n);\mathrm{conv}\,(\Bbb R^n))$, $\mathrm{conv}\,(\Bbb R^n)$  is a  metric space of nonempty convex compacts with Hausdorff metric.

 In this paper, the concept of Lyapunov is   applied for  more general problem i.e. study of stability in terms of two measures of  semiflows in space   $\mathrm{conv}\,(\Bbb R^n)$.

Let us recall a general definition of the (local) semiflow for an arbitrary complete metric space $(X,d)$.

Let $D$ be a set $\{0\}\times X\subset D\subset\Bbb R_+\times X$, such that $D\setminus(\{0\}\times X)$  is an open nonempty set in  space  $\Bbb R_+\times X$.

\begin{definition} The mapping $\mathfrak{F}:\,\,D\to X$, $t\ge 0$ is called a (local) semiflow in  metric space  $(X,d)$ if the following conditions are fulfilled:

(1) $\mathfrak{F}^0(u)=u$ for all $u\in X$;

(2) for all $(t_1,u_0)\in D$, $(t_2,u_0)\in D$ such that $(t_1+t_2,u_0)\in D$ the following equality holds
$$
\mathfrak{F}^{t_1}(\mathfrak{F}^{t_2}(u_0))=\mathfrak{F}^{t_1+t_2}(u_0).
$$

(3) for any $\varepsilon>0$ and $(t_0,u_0)\in D$ there exists a positive number $\dl=\dl(t_0,u_0,\varepsilon)$ such that  for all $(t,u_0)\in D$ the inequality $|t-t_0|<\dl$ implies the estimate  $\,\,d(\mathfrak{F}^{t}(u_0),\mathfrak{F}^{t_0}(u_0))<\varepsilon$;

(4) for any $\varepsilon>0$ and $(t_0,u_0)\in D$ there exists a positive number  $\dl=\dl(t_0,u_0,\varepsilon)$ such that  for all $(t_0,u)\in D$ the inequality $\,\,d(u,u_0)<\dl$ implies the estimate $d(\mathfrak{F}^{t_0}(u),\mathfrak{F}^{t_0}(u_0))<\varepsilon$;

If $D=\Bbb R_+\times X$, then the semiflow $\mathfrak{F}$ is called a global semiflow.
\end{definition}

The aim of this work is to study based on the general concept of  A.M. Lyapunov the stability in terms of two measures of global semiflow  $\mathfrak{F}$ in the metric space  $(X,d_H)$,  where $X$  is a closed subset in $\mathrm{conv}\,(\Bbb R^n)$.

We shall recall that  Lyapunov's concept involves the following steps:

(1)   selection of measures of initial and current variations in a phase space. So, the stability is considered in terms of these two measures;

(2)   selection of auxiliary  Lyapunov function;

(3)   verification of stability conditions, which involves the calculation of changes of auxiliary Lyapunov function along the orbit of the semiflow and construction  of  comparison system.

We shall note that   the stability problem statement of semiflow $\mathfrak{F}$ in  the space
$\mathrm{conv}\,(\Bbb R^n)$ generalizes the  problem statement, which was considered previously in \cite{devi}. Indeed, suppose that semiflow $\mathfrak{F}$ such that the mapping $t\to\mathfrak{F}^t(u)$  is differentiable by Hukuhara for all $t\in[0,\Omega^+(u))$. Then it is possible for all $u\in\mathrm{conv}\,(\Bbb R^n)$ to determine the generator of this semiflow
\begin{equation}
F(u)=\lim\limits_{t\to 0+}\frac{1}{t}(\mathfrak{F}^t(u)-u).
\end{equation}
Let $F:\,\mathrm{conv}\,(\Bbb R^n)\to\mathrm{conv}\,(\Bbb R^n)$ and suppose that $F$ satisfies the local Lipschitz condition, then the initial value problem for  the SDEs
\begin{equation}
D_Hu=F(u),\quad u(0)=u_0,
\end{equation}
where $u(t)\in\mathrm{conv}\,(\Bbb R^n)$, $t\ge 0$, $D_H$ is a Hukuhara derivative, generates a shift operator along trajectories, which coincides with the initial local semiflow $\mathfrak{F}$.

In  \cite{devi}   the conditions for stability and asymptotic stability of  stationary solution of SDEs \eqref{1-1} were established.
 In this case, the definition of notion ''stability'' have some difficulties because the function $\diam\!\mathfrak{F}^t(u)$ is a nondecreasing function of time $t\ge 0$, which leads to the fact that the solutions  of SDEs are usually characterized by instability.
 Therefore there is a need for modification of the classical definitions of stability. However such modifications  usually contain the conditions of  existence of the Hukuhara difference which are difficult to verify.

The proposed in this paper   problem statement and the  obtained  results allow us to overcome these difficulties, since for the general semiflow $\mathfrak{F}$ there is no need to require  the nondecreasing  on $t$ of function $\diam\mathfrak{F}^t(u)$.
 Another difference between the proposed statement of the problem is that we consider the stability  of the semiflow $\mathfrak{F}^t(u)$ in terms of two  measures.

 The choice of these measures is dictated by the nature of the geometric elements of the space $\mathrm{conv}\,(\Bbb R^n)$. Namely,  the functionals of the form $\,V_m[u,K]=V\underbrace{[u,...,u,\overbrace{K,...,K}^m]}_{n}$  will be considered as such measures. Here $V[u_1,...,u_n]$ is a functional of Minkowskij mixed volume, $K$ is a unit ball in space $\Bbb R^n$. These functionals have the known geometric meaning, for example, $V[u,...,u]$ is a volume of the convex body $u$, $nV[u,...,u,K]$ is a surface area of border $\partial u$ of the convex body $u$, etc.

The  obtained results are quite constructive and, in some cases, allow to obtain estimates of the elements of semiflow on a finite time interval, as well as the conditions of practical stability of semiflow on a finite time interval.

\section{Preliminaries}
%At first we shall recall some known notions and facts [3,4] from convex geometry and the theory of mixed volumes which will be necessary for the next statement.
Further we shall need the following notions and results from convex geometry and the theory of mixed volumes, see [3]--[4].
The space $\mathrm{conv}\,(\Bbb R^n)$ is a set of nonempty convex compacts with the Hausdorff metric
\begin{equation}
d_H(u,v)=\inf\{\varepsilon>0\,|\,u\subset v+\varepsilon K,\quad v\subset u+\varepsilon K\}.
\end{equation}
In the space $\mathrm{conv}\,(\Bbb R^n)$  the operations of addition and  nonnegative scalar multiplication    are defined:
\begin{equation}
u+v=\{x+y\,\,|\,x\in u,\,\,y\in v\},\quad \lambda u=\{\lm x\,\,|\,x\in u\}, \quad \lambda\ge 0.
\end{equation}
 The element $w \in \mathrm{conv}\,(\Bbb R^n)$  is called the Hukuhara difference for elements $u,\,v\in \mathrm{conv}\,(\Bbb R^n) $, if  $u=v+w$.
  The Hukuhara difference  of two elements from $ \mathrm{conv}\,(\Bbb R^n) $ is not always defined.

  The concept of  Hukuhara  difference allows to determine  the notion of the  Hukuhara derivative for certain mappings
  $T\to\mathrm{conv}\,(\Bbb R^n)$, $T$ is an open set in $\Bbb R$.

 Note also that from the analytical point of view it is conveniently to investigate the SDEs  based on  the theory of support functions.
  Recall that  every nonempty convex compact  $u$ we can associate with its  support function $h_u\,\,:\Bbb R^n\to\Bbb R$ where
  \begin{equation}
h_u(p)=\sup\limits_{x\in u}(x,p).
\end{equation}
Note that $u\to h_u(.)$ is  isomorphic and isometric map  $\mathrm{conv}\,(\Bbb R^n)\to C(S^{n-1})$, i.e.
\begin{equation}\gathered
u+v\to h_u(p)+h_v(p),\quad \lm u\to h_{\lm u}(p),\;\; \lm\ge 0,\\ d_H(u,v)=\|h_u(p)-h_v(p)\|_{C(S^{n-1})}.
\endgathered
\end{equation}
This fact allows us to identify nonempty compact convex sets and their support functions that  further  will not be specially stipulated.

We now recall some basic concepts and results of the theory of mixed volumes of H. Minkowskij. Let $u_k\in\mathrm{conv}\,(\Bbb R^n)$,  $\lm_k$ are nonnegative numbers, $k=\overline{1,m}$,
$u=\sum\limits_{k=1}^m\lm_ku_k\in\mathrm{conv}\,(\Bbb R^n)$.

  Minkowskij has shown that  a volume $V[u]$ of convex body $u$ is a homogeneous polynomial of degree $n$ relative to the variables  $\lm_k$
\begin{equation}
V[u]=\sum\limits_{k_1,...,k_n}V_{k_1,...,k_n}\lambda_{k_1}\dots\lambda_{k_n},
\end{equation}
where the sum is taken over all indices $k_1$, $...$, $k_n$ which vary independently over all values from $1$ to $m$.
At the same time the coefficients of $V_{k_1,...,k_n}$ are determined  so that they do not depend on the order of the indices.

 One can show that $V_{k_1,...,k_n}$ depend only on the bodies $u_{k_1}$, $...$, $u_{k_n}$. Therefore it is natural to write it in the form $V[u_{k_1},..., u_{k_n}]$.These coefficients are called the mixed volumes.

The functional $V[u_1,...,u_n]$ has the following properties:

(1)  $V[u_1,...,u_n]$ is  additive and positively homogeneous respect to each variable, i.e. for all $\lambda^{\prime}$, $\lambda^{\prime\prime} \in \Bbb R_+$, $u_k\in\mathrm{conv}\,(\Bbb R^n)$, $w\in\mathrm{conv}\,(\Bbb R^n)$
\begin{equation}\gathered
V[u_1,...,\lambda^{\prime}w+\lambda^{\prime\prime}u_k,...,u_n]=\lambda^{\prime}V[u_1,...,w,...,u_n]\\ \qquad\qquad\qquad\qquad\qquad\qquad \;+\lambda^{\prime\prime}V[u_1,...,u_k,...,u_n];
\endgathered
\end{equation}

(2)  $V[u_1,...,u_n]$ is   a translation invariant and invariant with respect to permutation of arguments, as well as  a continuous respect to the totality of variables  \cite{Lihtv}.

From these properties  the Steiner formula   is derived
\begin{equation}\label{2}
V[u_1+\varrho u_2]=\sum\limits_{k=0}^n C_n^k\varrho^kV_k[u_1,u_2],\quad\varrho\in\Bbb R_+,
\end{equation}
where $\,V_k[u_1,u_2]=V\underbrace{[u_1,...,u_1,\overbrace{u_2,...,u_2}^k]}_{n}$.

From \eqref{2} we have
\begin{equation*}
 nV_1[u_1,u_2]=\lim\limits_{\varrho\to 0}\frac{V[u_1+\varrho u_2]-V[u_1]}{\varrho}.
\end{equation*}
From  A.D. Alexandrov inequality
\begin{equation}\label{3a}
V^2[u_1,...,u_{n-1},u_n]\ge V[u_1,...,u_{n-2},u_{n-1},u_{n-1}]V[u_1,...,u_{n-2},u_{n},u_{n}]
\end{equation}
 it follows the inequalities for functionals
 $V_k[u_1,u_2]$
\begin{equation}\label{fe}
V_k^2[u_1,u_2]\ge V_{k-1}[u_1,u_2]V_{k+1}[u_1,u_2].
\end{equation}
From \eqref{fe} we obtain the estimate
\begin{equation}\label{4}
 V_k[u_1,u_2]\ge V^{(n-k)/n}[u_1]V^{k/n}[u_2].
\end{equation}
A special case of the inequality  \eqref{4} with $k = 1$ is the isoperimetric  Brunn--Minkowskij inequality.

Since the  Hukuhara difference between  two elements of the  space $\mathrm{conv}\,(\Bbb R^n)$, as well as  the   derivative of mapping $(\al,\be)\to\mathrm{conv}\,(\Bbb R^n)$  are not always defined so there is a need to embed  the space $\mathrm{conv}\,(\Bbb R^n)$ into a corresponding Banach space, so that the  Hukuhara difference  of any two elements of the space has always been defined as the element of this wider space and therefore the notion of derivative  would be applicable to a wider class of mappings.

Such embedding was realized in 1937 in the work of Academician A.D. Alexandrov \cite{alex1}. Similar structures are also given in more recent works \cite{reds}--\cite{polovinkin}.
Before we give the appropriate structure let us make the following remark. The space $\mathrm{conv}\,(\Bbb R^n)$ is isometrically and isomorphically embedded as a wedge into the space of continuous functions $C(S^{n-1})$ on the unit sphere $S^{n-1}$.

Such embedding is realised by correspondence of each element $u\in\mathrm{conv}\,(\Bbb R^n)$ to its support function.
Therefore, further  the elements of the space $\mathrm{conv}\,(\Bbb R^n)$ will be identified with their support functions.

Let us describe the embedding of the space $\mathrm{conv}\,(\Bbb R^n)$ into a linear normed space $\mathcal{A}_n$ such that in this space for any two elements the operation of difference of these elements  is feasible.

Consider a set $\mathrm{conv}\,(\Bbb R^n)\times\mathrm{conv}\,(\Bbb R^n)$ and introduce on this set a binary equivalence relation $\rho$
\begin{equation*}
\gathered
(u,v)\rho(w,z)\equiv (u+z=v+w).
\endgathered
\end{equation*}
Let  $\mathcal{A}_n=\mathrm{conv}\,(\Bbb R^n)\times\mathrm{conv}\,(\Bbb R^n)/\rho$. In the space $\mathcal{A}_n$ the operations of addition and    multiplication by a scalar $\lm\in\Bbb R$ are introduced. If $[(u,v)]\in\mathcal{A}_n$, $[(w,z)]\in\mathcal{A}_n$, then
\begin{equation*}
\gathered
  \lm[(u,v)]=\begin{cases} [(\lm u,\lm v)],\quad\lm\ge 0\\
[(|\lm| v,|\lm| u)],\quad\lm\le 0
\end{cases},\,\,\,
[(u,v)]+[(w,z)]=[(u+w,v+z)].
\endgathered
\end{equation*}

These operations are correctly defined, and the original space $\mathrm{conv}\,(\Bbb R^n)$ is isomorphically embedded in $\mathcal{A}_n$ by the rule $\,\mathrm{conv}\,(\Bbb R^n)\ni u\to[(u,0)]\in\mathcal{A}_n$.

  In $\mathcal{A}_n$ we can introduce the norm $\|[(u,v)]\|_{\mathcal{A}_n}=d_H(u,v)$.

 This definition is correct, and the corresponding embedding $\mathrm{conv}\,(\Bbb R^n)\to\mathcal{A}_n$ is  the isometric embedding of a metric space $\mathrm{conv}\,(\Bbb R^n)$ into the metric space $\mathcal{A}_n$, in which  metric is generated by the introduced norm $\|.\|_{\mathcal{A}_n}$.

By terminology of the monograph \cite{krasn}, the space $\mathrm{conv}\,(\Bbb R^n)$ is a wedge in a  linear normed  space  $\mathcal{A}_n$.

Note also that the space $\mathcal{A}_n$  is not  a complete, but in \cite{alex1} was proved that its  completion coincides with $C(S^{n-1})$.

For further discussion it is necessary to extend the effect of the functional $V[u_1, ..., u_n]$ on the elements of space $C(S^{n-1})$.
We can do this in an obvious way: at first by using the  multilinearity property of this functional it extends   to the space $\,\mathcal{A}_n$, and then, using the extension by continuity, --- on the elements of the space $C(S^{n-1})$.
  The explicit formulas of such  extension are given in  \cite{alex1}.

 Consider the functional $V[u_1,...,u_{n-1},z]$ for fixed $u_i\in\mathrm{conv}\,(\Bbb R^n)$. 
 
 $V[u_1,...,u_{n-1},z]$ is a linear and continuous functional in a Banach space $C(S^{n-1})$, so  it can be represented as a Stieltjes--Radon integral of a continuous function $z\in C(S^{n-1})$  by a uniquely defined additive set function on the unit sphere $S^{n-1}$. The functional $V[u_1,...,u_{n-1},z]$ is completely determined by specifying of convex bodies $u_i\in\mathrm{conv}\,(\Bbb R^n)$, so we can state that
\begin{equation}\label{2111}
 V[u_1,...,u_{n-1},z]=\frac{1}{n}\int\limits_{S^{n-1}}z(p)F[u_1,...,u_{n-1};d\omega], \quad z\in C(S^{n-1}),
\end{equation}
where $F[u_1,...,u_{n-1};d\omega]$ is a  function of set $\omega$ on the unit sphere  $S^{n-1}$, which is uniquely determined by the convex compacts  $u_i\in\mathrm{conv}\,(\Bbb R^n)$. This function is called a mixed superficial function of convex compacts  $u_i\in\mathrm{conv}\,(\Bbb R^n)$. One can show that
$$
F[u_1,...,u_{n-1};d\omega]\ge 0.
$$

\section{Statement of the problem and the general  stability theorems}%3
Next we formulate the definition of stability in terms of two  measures of semiflow $\mathfrak{F}$. Let $h_0\in C(\mathrm{conv}\,(\Bbb R^n);\Bbb R_+)$, $h\in C(\mathrm{conv}\,(\Bbb R^n);\Bbb R_+)$ be  the measures of initial and current variations and
$$
\inf\limits_{X\in\mathrm{conv}\,(\Bbb R^n)}h_0[X]=\inf\limits_{X\in\mathrm{conv}\,(\Bbb R^n)}h[X]=0.
$$

\begin{definition}  The global semiflow $\mathfrak{F}$ is

(1) $(h_0,h)$-stable, if for any $\varepsilon>0$ there exists $\dl=\dl(\varepsilon)>0$  such that the inequality $h_0[u_0]<\dl$ implies the estimate  $h[\mathfrak{F}^t(u)]<\varepsilon$ for all $t\ge 0$;

(2) asymptotically $(h_0,h)$-stable, if it $(h_0,h)$-stable and there exists $\rho>0$  such that from inequality $h_0[u_0]<\rho$ follows that $\lim\limits_{t\to+\infty}h[\mathfrak{F}^t(u)]=0$;

(3) practical $(\lambda,A,T)$-stable ($0<\lambda<A$) in terms of two  measures $(h_0,h)$ on finite time interval, if inequality $h_0[u_0]<\lambda$ implies the estimate $h[\mathfrak{F}^t(u)]<A$ for all $t\in[0,T]$.

\end{definition}

The aim of this work is to study the stability  in terms of two measures of semiflows $\mathfrak{F}$ in space $\mathrm{conv}\,(\Bbb R^n)$. At first we shall formulate the general result, which will be specified in the following sections of this paper.

For further discussion  we need the following assumptions.
\begin{assumption}
 There exist the functionals $W_i\,\,:\mathrm{conv}\,(\Bbb R^n)\to\Bbb R_+$, $i=0,...,k$ which are differentiable along the orbit of semiflow $\mathfrak{F}$, the functions $f_i\in C(\Bbb R^{i+1};\Bbb R)$, $\psi_i\in C(\Bbb R^{i+1};\Bbb R)$ and the functions $a(.)$, $b(.)$   of Hahn class such that

 (1)
\begin{equation*}
\gathered
\frac{dW_i[\mathfrak{F}^t(u)]}{dt}\Big|_{t=0}\le f_i(W_0[u],W_1[u],...,W_i[u])\qquad\qquad\qquad\;\\ \qquad\qquad\qquad\qquad \qquad +\,\psi_{i}(W_0[u],W_1[u],...,W_i[u])W_{i+1}[u],\;\;i=\overline{0,k-1},\\
\frac{dW_k[\mathfrak{F}^t(u)]}{dt}\Big|_{t=0}\le f_k(W_0[u],W_1[u],...,W_k[u]);\qquad \qquad \qquad  
\endgathered
\end{equation*}

(2) the right sides of comparison system
\begin{equation}\label{s*}
\gathered
 \frac{d\xi_i}{dt}=f_i(\xi_0,...,\xi_i)+\psi_{i}(\xi_0,\xi_1,...,\xi_i)\xi_{i+1},\;\; i=\overline{0,k-1},\\
\frac{d\xi_k}{dt}=f_k(\xi_0,...,\xi_k)\qquad \qquad \qquad \qquad \qquad \qquad \qquad \quad\;\;
\endgathered
\end{equation}
satisfy the conditions of T. Wazewskij;

(3) the inequalities
$$
\max\limits_{i=\overline{0,k}} W_i[u]\le b(h_0[u]),\quad W_0[u]\ge a(h[u])
$$
are fulfilled.
\end{assumption}

Note that for vector $\xi=(\xi_0, \xi_1, \dots, \xi_k)$ we shall use the following norm $\|\xi\|_{\infty}=\max\limits_{i=\overline{0,k}}|\xi_i|$.
\begin{definition} The solution $\xi_i=0$, $i=0,1,...,k$ of comparison system is

(1) $\xi_0$-stable in cone $\Bbb R_+^{k+1}$ if and only if for any $\varepsilon>0$ there exists $\dl=\dl(\varepsilon)>0$ such that for all  $\xi(0)>0$ the inequality $\|\xi(0)\|_{\infty}<\dl$ implies the estimate $\xi_0<\varepsilon$ for all $t\ge 0$;

(2) asymptotically $\xi_0$-stable in cone $\Bbb R_+^{k+1}$ if and only if it is $\xi_0$-stable  and there exists $\rho>0$ such that for all  $\xi(0)>0$ from inequality $\|\xi(0)\|_{\infty}<\rho$ follows  that $\lim\limits_{t\to+\infty}\xi_0(t)=0$.
\end{definition}

\begin{theorem}\label{t31}
Assume that  comparison system \eqref{s*} has a trivial solution  $\xi_i=0$, $i=0,...,k$ which is

(1) $\xi_0$-stable in cone $\Bbb R_+^{k+1}$;

(2) asymptotically  $\xi_0$-stable in cone $\Bbb R_+^{k+1}$.

Then the semiflow $\mathfrak{F}$ is

(1) $(h_0,h)$-stable;

(2) asymptotically $(h_0,h)$-stable.
\end{theorem}
\begin{proof} Let $\varepsilon>0$,  then  under hypothesis of theorem  there exists a number  $\Delta(a(\varepsilon))$ such that the inequality $\|\xi(0)\|_{\infty}<\Delta(a(\varepsilon))$ implies the estimate $0<\xi_0(t;\xi(0))<a(\varepsilon)$ for all $t\ge 0$. Let $h_0[u]<b^{-1}(\Delta(a(\varepsilon)))$. From  differential inequality  theorem \cite{rush}, it follows that
$$
W_k[\mathfrak{F}^t(u)]\le\xi_k(t;\xi(0)),\quad t\ge 0,
$$
whenever $\xi_i(0)=W_i[u]$, $i=0,...,k$. In this case, $\|\xi(0)\|_{\infty}=\max\limits_{i=\overline{0,k}}W_i[u]\le b(h_0[u])<\Delta(a(\varepsilon))$, therefore
$$
a(h[\mathfrak{F}^t(u)])\le W_0[\mathfrak{F}^t(u)]\le\xi_0(t;\xi(0))<a(\varepsilon),\quad t\ge 0.
$$
Hence it follows that $h[\mathfrak{F}^t(u)]<\varepsilon$, $t\ge 0$. The stability in terms of two measures $(h_0,h)$ is proved. Asymptotic stability is proved similarly. This completes the proof.
\end{proof}

Similarly we can prove the following result.

\begin{theorem}\label{t32} Assume that   for comparison system \eqref{s*} the following inequality  holds
$$
\xi_0(T;b(\lambda)e_0)<a(A), \;\;e_0=(1,1,\dots,1)^T.
$$

Then the semiflow $\mathfrak{F}$ is  practical $(\lambda,A,T)$-stable  in terms of two  measures $(h_0,h)$ on finite time interval.
\end{theorem}

\section{Stability in terms of    two measures for a certain semiflow.}
In this section we shall describe  the semiflow which will  be studied in detail in this paper.

Consider the ordinary differential equation
\begin{equation*}
\frac{dx}{dt}=Ax,
\end{equation*}
where $x\in\Bbb R^n$, $A\in\mathfrak{L}(\Bbb R^n)$, $\mathfrak{L}(\Bbb R^n)$ is a Banach algebra of linear continuous operators in $\Bbb R^n$.  This equation generates   a semiflow $\{e^{At}\}_{t\ge 0}$ in the space $\Bbb R^n$ and  it can be  naturally extends  to  a semiflow  in the metric space  $\mathrm{conv}\,(\Bbb R^n)$
\begin{equation*}
\mathrm{conv}\,(\Bbb R^n)\ni u_0\to e^{At}u_0\in\mathrm{conv}\,(\Bbb R^n).
\end{equation*}
  Let  $\exp\{\mathcal{A}t\}$, $t\ge 0$   denote such semiflow. The generator of this semigroup be of the form
\begin{equation*}
 \mathcal{A}u=\lim\limits_{t\to 0+}\frac{\exp\{\mathcal{A}t\}u-u}{t}=(\nabla_ph_u(p),A^Tp),
\end{equation*}
where the right side of the expression should be understood as the derivative with respect to the direction which guarantees the existence of the right side of the expression for any support function, i.e. a definitional domain of the generator $\mathfrak{D}(\mathcal{A})=\mathrm{conv}\,(\Bbb R^n)$.
 The explicit effect of the operator $\exp\{\mathcal{A}t\}$  on the elements of the space $\mathrm{conv}\,(\Bbb R^n)$ is expressed as follows

\begin{equation}\label{exop}
\exp\{\mathcal{A}t\}u=h_u(e^{A^T t}p)=\|e^{A^T t}p\|h_u\Big(\frac{e^{A^T t}p}{\|e^{A^T t}p\|}\Big).
\end{equation}

The  formula \eqref{exop} allows to extend the operator $\exp\{\mathcal{A}t\}$  to a continuous linear operator in a Banach space $C(S^{n-1})$.

Further we shall need the following  inequalities
\begin{enumerate}
  \item if $N$ and $\alpha$ are constants such that  $\|e^{A t}\|\le Ne^{\alpha t}$, $t\ge 0$, then
\begin{equation*}
\|\exp\{\mathcal{A}t\}f\|_{C(S^{n-1})}\le Ne^{\al t}\|f\|_{C(S^{n-1})},\quad t\ge 0,\quad f\in S^{n-1}.
\end{equation*}
  \item \begin{equation*}
\|\exp\{\mathcal{A}t\}f\|_{C(S^{n-1})}\le e^{\|A\|t}\|f\|_{C(S^{n-1})},\quad t\ge 0,\quad f\in S^{n-1}.
\end{equation*}
  \item \begin{equation*}
\|\exp\{\mathcal{A}t\}u-u\|_{C(S^{n-1})}\le (e^{\|A\|t}-1)\|u\|_{C(S^{n-1})},\quad t\ge 0,\quad u\in\mathrm{conv}\,(\Bbb R^n).
\end{equation*}
\end{enumerate}

Indeed, the first property follows from the inequalities
\begin{equation*}
\|\exp\{\mathcal{A}t\}f\|_{C(S^{n-1})}\le\|e^{A^Tt}\|\|f\|_{C(S^{n-1})}\le Ne^{\alpha t}\|f\|_{C(S^{n-1})}.
\end{equation*}
The second property is proved similarly.

To prove the third we shall use the property of Lipschitz  for support function
\begin{equation*}
\gathered
\|\exp\{\mathcal{A}t\}u-u\|_{C(S^{n-1})}=\|h_u(e^{A^Tt}p)-h_u(p)\|_{C(S^{n-1})}\\ \qquad \le
\|h_u\|_{C(S^{n-1})}\|e^{A^Tt}-I\|\\ \le
(e^{\|A\|t}-1)\|u\|_{C(S^{n-1})}.
\endgathered
\end{equation*}

 Let $\varphi\,\,:\Bbb R_+\to\Bbb R_+$, $F\,\,:\Bbb R_+\times\mathrm{conv}\,(\Bbb R^n)\to\Bbb R^n$.
Assume, that the function $\varphi$ and the mapping  $F$ satisfy a local Lipschitz condition,
  i.e. for any $s_0>0$ there exist  constants $\delta>0$ and $H>0$  such that for any $s_i\in(s_0-\delta,\,s_0+\delta)$, $i=1,2$ the following inequality holds
\begin{equation}\label{41a}
|\varphi(s_2)-\varphi(s_1)|\le H|s_2-s_1|,
\end{equation}
and for any point $(s_0,u_0)\in \mathbb{R}_+\times \mathrm{conv}\,(\Bbb R^n) $ there exists its neighborhood $U$ and a constant $L>0$ such that for all $(s_i,u_i)\in U$, $i=1,2$  the following inequality holds
\begin{equation}\label{41b}
\|F(s_2,u_2)-F(s_1,u_1)\|_{C(S^{n-1})}\le L(\|u_2-u_1\|_{C(S^{n-1})}+|s_2-s_1|).
\end{equation}

Note that from the formula of Steiner  it follows that for any ball $K_r(u_0)\subset\mathrm{conv}\,(\Bbb R^n)$, $r>0$ there exists a  constant
$L=L(r)>0$ such that for all $u_i\in K_r(u_0)$, $i=1,2$ we have
\begin{equation}\label{41c}
|V[u_2]-V[u_1]|\le L\|u_2-u_1\|_{C(S^{n-1})}.
\end{equation}
%Therefore,  (2) implies that the mapping $F(V[u],u)$ also satisfies the local Lipschitz condition.

Consider the integral equation
\begin{equation}\label{41}
\gathered
u(t)=\exp\Big\{\mathcal{A}\int\limits_{0}^t\varphi(V[u(s)])\,ds\Big\}u_0\qquad \qquad \qquad  \\ \quad \quad \qquad  \;\; +\int\limits_{0}^t\exp\Big\{\mathcal{A}\int\limits_{s}^t\varphi(V[u(\tau)])\,d\tau\Big\} F(V[u(s)],u(s))\,ds,
\endgathered
\end{equation}
where $u_0\in\mathrm{conv}\,(\Bbb R^n)$.

\begin{theorem}\label{t41} For  a sufficiently small number $T^*>0$ there exists a unique mapping $u\in C^1([0,T^*);\mathrm{conv}\,(\Bbb R^n))$ which satisfies the integral equation \eqref{41}.
\end{theorem}

\begin{proof} Let us define a metric space
$
C_T=C([0,T];\mathrm{conv}\,(\Bbb R^n))
$
with metric $\varrho_T(u_1(.),u_2(.))=\max\limits_{t\in[0,T]}\|u_1(t)-u_2(t)\|_{C(S^{n-1})}$.

From inequalities \eqref{41a}--\eqref{41c} it follows that there exist a neighbourhood $U$ of point $u_0\in \mathrm{conv}\,(\Bbb R^n)$ and  constants $L>0$, $L_1>0$ such that for all $u_i\in U$  the inequalities
$$
\|F(V[u_2],u_2)-F(V[u_1],u_1)\|_{C(S^{n-1})}\le L_1\,\|u_2-u_1\|_{C(S^{n-1})},
$$
$$
|V[u_2]-V[u_1]|\le L\,\|u_2-u_1\|_{C(S^{n-1})}
$$
are  fulfilled.

Since $\varphi(s)$ is a locally Lipschitz, then there exist a neighborhood $(V[u_0]-\delta, V[u_0]+\delta)$  and a constant $H>0$ such that for all $s_i\in (V[u_0]-\delta, V[u_0]+\delta) $, $i=1,2$ the inequality \eqref{41a} is fulfilled.

Choose $r>0$ such that $Lr< \delta$ and $K_r(u_0)\subset U$.

Define the operator $\mathfrak{G}: C_T\to C_T$
\begin{equation*}
\gathered
(\mathfrak{G}u)(t)=\exp\Big\{\mathcal{A}\int\limits_{0}^t\varphi(V[u(s)])\,ds\Big\}u_0\qquad\qquad\qquad\qquad\qquad\qquad\qquad\\
\qquad\quad\;\; +\int\limits_{0}^t\exp\Big\{\mathcal{A}\int\limits_{s}^t\varphi(V[u(\tau)])\,d\tau\Big\} F(V[u(s)],u(s))\,ds,\quad t\in[0,T].
\endgathered
\end{equation*} In space $C_T$ we consider a ball $B_{r}(u_0)$. We shall show that the number $T$ can be chosen so small that  $\mathfrak{G}$ maps the ball $B_{r}(u_0)$  in oneself. From estimates
\begin{equation*}
\gathered
\|(\mathfrak{G}u)(t)-u_0\|_{C(S^{n-1})}\le\Big\|\exp\Big\{\mathcal{A}\int\limits_{0}^t\varphi(V[u(s)])\,ds\Big\}u_0-u_0\Big\|_{C(S^{n-1})}\\
\;\;\;\;+\int\limits_{0}^t\Big\|\exp\Big\{\mathcal{A}\int\limits_{s}^t\varphi(V[u(\tau)])\,d\tau\Big\}F(V[u(s)],u(s))\Big\|_{C(S^{n-1})}\,ds\\
\qquad\le(e^{(\varphi(V[u_0])+LHr)t}-1)\|u_0\|\\
+\int\limits_{0}^te^{(\varphi(V[u_0])+LHr)(t-s)}(\|F(V[u_0],u_0)\|_{C(S^{n-1})}+L_1r)\,\,ds\\
=(e^{\be t}-1)\Big(\|u_0\|_{C(S^{n-1})}+\frac{\|F(V[u_0],u_0)\|_{C(S^{n-1})}+L_1r}{\be}\Big)
\endgathered
\end{equation*}
it follows the inequality
\begin{equation}
\varrho_T((\mathfrak{G}u)(.),u_0)\le (e^{\be T}-1)\Big(\|u_0\|_{C(S^{n-1})}+\frac{\|F(V[u_0],u_0)\|_{C(S^{n-1})}+L_1r}{\be}\Big).
\end{equation}
%Here $L_1=L(L_K+1)$, $\be=\varphi(V[u_0])+LHr$, $H$ is  a Lipschitz constant for the function $\varphi$.
Here  $\be=\varphi(V[u_0])+LHr$,  $H$ is  a Lipschitz constant for the function $\varphi$.

Obviously that there exists $T_1>0$ such that for all $T\in(0,T_1]$ the following inequality holds
\begin{equation}
(e^{\be T}-1)(\|u_0\|_{C(S^{n-1})}+\frac{\|F(V[u_0],u_0)\|_{C(S^{n-1})}+L_1r}{\be})\le r,
\end{equation}
i.e. $\mathfrak{G}B_r(u_0)\subset B_r(u_0)$.  Next we show  that   decreasing of number $T$ can ensure that the operator $\mathfrak{G}$ is a contraction operator. If $u_i\in B_r(u_0)$, $i=1,2$ then
\begin{equation}
\gathered
(\mathfrak{G}u_1)(t)-(\mathfrak{G}u_2)(t)=\Big(\exp\{\mathcal{A}\int\limits_{0}^t\varphi(V[u_1(s)])\,ds\}u_0\\ \qquad\qquad\qquad\qquad\;\,
-\exp\{\mathcal{A}\int\limits_{0}^t\varphi(V[u_2(s)])\,ds\}u_0\Big)\\+
\int\limits_{0}^t\Big[\exp\{\mathcal{A}\int\limits_{s}^t\varphi(V[u_1(\tau)])\,d\tau\}F(V[u_1(s)],u_1(s))\\
\endgathered
\end{equation}
\begin{equation*}
\gathered
-\exp\{\mathcal{A}\int\limits_{s}^t\varphi(V[u_2(\tau)])\,d\tau\}F(V[u_2(s)],u_2(s))\Big]\,ds\\ \quad\;=
\exp\{\mathcal{A}\int\limits_{0}^t\varphi(V[u_2(s)])\,ds\}\Big(\exp\{\mathcal{A}\int\limits_{0}^t(\varphi(V[u_1(s)])u_0\\
-\varphi(V[u_2(s)]))\,ds\}-u_0\Big)\qquad\qquad\qquad\qquad\qquad\quad
\endgathered
\end{equation*}
\begin{equation*}
\gathered
+
\int\limits_{0}^t\exp\{\mathcal{A}\int\limits_{s}^t\varphi(V[u_1(\tau)])\,d\tau\}(F(V[u_1(s)],u_1(s))-F(V[u_2(s)],u_2(s)))\,ds\\+
\int\limits_{0}^t\exp\{\mathcal{A}\int\limits_{s}^t\varphi(V[u_2(\tau)])\,d\tau\}\Big(
\exp\{\mathcal{A}\int\limits_{s}^t(\varphi(V[u_1(\tau)])\qquad\qquad\qquad\qquad\;\;\\
-\varphi(V[u_2(\tau)])\,d\tau\}-I\Big)F(V[u_2(s)],u_2(s))\,ds.\qquad\qquad\qquad\qquad\qquad\qquad\;\;
\endgathered
\end{equation*}
Let us estimate for $t\in [0,\,T]$ each summand by norm separately, taking into account the obvious inequality  $e^x\le 1+xe^x$, $x\ge 0$. So we obtain
\begin{equation}
\gathered
\Big\|\exp\{\mathcal{A}\int\limits_{0}^t\varphi(V[u_1(s)])\,ds\}u_0-
\exp\{\mathcal{A}\int\limits_{0}^t\varphi(V[u_2(s)])\,ds\}u_0\Big\|_{C(S^{n-1})}\\
\le\Big\|\exp\{\mathcal{A}\int\limits_{0}^t\varphi(V[u_2(s)])\,ds\}\Big\|_{C(S^{n-1})}\qquad\qquad\qquad\qquad\qquad\qquad\\
\times
\Big\|\exp\{\mathcal{A}\int\limits_{0}^t(\varphi(V[u_1(s)])-\varphi(V[u_2(s)])\,ds\}u_0-u_0\Big\|_{C(S^{n-1})}\qquad\quad\\
\le e^{\be T}(e^{\|A\|HL\varrho_T(u_1,u_2)T}-1)
\le\|A\|HLTe^{(\be+2r\|A\|HL)T}\varrho_T(u_1,u_2).
\endgathered
\end{equation}
\begin{equation}
\gathered
\Big\|\int\limits_{0}^t\exp\{\mathcal{A}\int\limits_{s}^t\varphi(V[u_1(\tau)])\,d\tau\}(F(V[u_1(s)],u_1(s))-F(V[u_2(s)],u_2(s)))\,ds\Big\|_{C(S^{n-1})}\\
\le\int\limits_{0}^te^{\be(t-s)}\,ds L_1\varrho_T(u_1,u_2)\le \frac{L_1}{\be}(e^{\be T}-1)\varrho_T(u_1,u_2).
\endgathered
\end{equation}
\begin{equation}
\gathered
\Big\|\int\limits_{0}^t\exp\{\mathcal{A}\int\limits_{s}^t\varphi(V[u_2(\tau)])\,d\tau\}\Big(
\exp\{\mathcal{A}\int\limits_{s}^t(\varphi(V[u_1(\tau)])-\varphi(V[u_2(\tau)])\,d\tau\}-I\Big)\\
\times F(V[u_2(s)],u_2(s))\,ds\Big\|_{C(S^{n-1})}\qquad\qquad\qquad\qquad\qquad\qquad \qquad\qquad\qquad\qquad\quad\\
\le \int\limits_{0}^t e^{\be(t-s)}(e^{\|A\|LH\varrho_T(u_1,u_2)(t-s)}-1)\,ds\Big(\|F(V[u_0],u_0)\|_{C(S^{n-1})}+L_1r\Big)\qquad\qquad\\\le
\|A\|LH\varrho_T(u_1,u_2)\int\limits_{0}^t(t-s) e^{(\be+2r\|A\|LH)(t-s)}\,ds(\|F(V[u_0],u_0)\|_{C(S^{n-1})}+L_1r)\\=\Big(\frac{Te^{\eta T}}{\eta}-\frac{1}{\eta^2}(e^{\eta T}-1)\Big)LH\|A\|(\|F(V[u_0],u_0)\|_{C(S^{n-1})}+L_1r)\varrho_T(u_1,u_2).\qquad
\endgathered
\end{equation}
Here $\eta=\be+2rLH\|A\|$.

From these estimates we obtain the inequality
\begin{equation}
\varrho_T(\mathfrak{G}u_1,\mathfrak{G}u_2)\le\gm(T)\varrho_T(u_1,u_2),
\end{equation}
where
\begin{equation}
\gathered
\gm(T)=\|A\|HLTe^{\eta T}+\frac{L_1}{\be}(e^{\be T}-1)\\+\Big(\frac{Te^{\eta T}}{\eta}-\frac{1}{\eta^2}(e^{\eta T}-1)\Big)LH\|A\|(\,\|F(V[u_0],u_0)\|_{C(S^{n-1})}+L_1r).
\endgathered
\end{equation}
 If $T\to 0$ then $\gm(T)\to 0$, so there exists a number $T_2>0$ such that for $T\in[0,T_2)$ and $\gm(T)<1$ the operator $\mathfrak{G}$ is a contraction operator which maps the ball $B_r(u_0)$  in oneself.
Denote  $T^*<\min\{T_1,T_2\}$. Applying the  Banach fixed-point theorem, we conclude that there exists  a unique fixed point $u^*\in B_r(u_0)\subset C_{T^*}$, i.e. $\mathfrak{G}u^*=u^*$. Thus the integral equation \eqref{41} has a solution $u^*(t)$ which satisfies the condition $u(0)=u_0$. This completes the proof of theorem.
\end{proof}
 Next we  consider two solutions $u_1(t)$ and $u_2(t)$ of the integral equation \eqref{41}, which are defined on a half-intervals $[0,r_1)$ and  $[0,r_2)$  respectively.

 Let us show that $u_1(t)\equiv u_2(t)$ for all $t\in[0,\min(r_1,r_2))$. Consider a set $\mathcal T=\{t>0,\,| u_1(t)=u_2(t)\}$, which is  nonempty, since $[0,T^*]\subset\mathcal T$ and closed, due to the continuity of these solutions.

Denote  $\tau^*=\sup\mathcal T\in\mathcal T$, $\tau^*<\min(r_1,r_2)$. Consider a solution $\widetilde{u}(t)$ of integral equation \eqref{41} with initial condition $\widetilde{u}(0)=u_1(\tau^*)=u_2(\tau^*)$. It is easy to prove that $\widetilde{u_1}(t)=u_1(t+\tau^*)$ and $\widetilde{u_2}(t)=u_1(t+\tau^*)$  are solutions of integral equation \eqref{41}  on the interval  $t\in[0,T^*]$, $T^*+\tau^*<\min(r_1,r_2)$.

Since $\widetilde{u}_1(0)=\widetilde{u}_2(0)$, then by the above, $T^*$ can be chosen so small that $\widetilde{u}_1(t)=\widetilde{u}_2(t)$, for all $t\in[0,T^*]$. Therefore $u_1(t)=u_2(t)$ for all $t\in[\tau^*,\tau^*+T^*]$, which contradicts the choice of $\tau^*$. So we have $\tau^*=\min(r_1,r_2)$.

From this assertion  it is  easy to prove that there exists a  unique solution of the integral equation \eqref{41} which is  nonextensible and defined on a maximal half-interval  $[0,\Omega^+(u_0))$.

 Let define a set $D=\bigcup\limits_{u_0\in\mathrm{conv}\,(\Bbb R^n)}[0,\Omega^+(u_0))\times\{u_0\}$ and for any $(t_0,u_0)\in D$  we set that $\mathfrak{F}^{t}(u_0)=u(t)$.

\begin{theorem}\label{t42} The mapping $\mathfrak{F}\,\,:D\to\mathrm{conv}\,(\Bbb R^n)$ is a local semiflow.
\end{theorem}

\begin{proof} It should be noted that the property (1) is obvious, and the property (2) (semigroup property) can be verified by direct calculation.

Let us prove the property (3) of the local semiflow. First of all, we establish the following result: for any $\varepsilon>0$ and $u_0\in\mathrm{conv}\,(\Bbb R^n)$ there exists $\delta=\delta(\varepsilon,u_0)>0$  such that the inequality $0<t<\delta$ implies the estimate $\|\mathfrak{F}^t(u_0)-u_0\|_{C(S^{n-1})}<\varepsilon$. Indeed, $\mathfrak{F}^t(u_0)=u(t)$  is a solution of integral equation  \eqref{41}. Next, we choose $\dl<T^*$, then for $0\le t<\dl$  the following estimate holds
\begin{equation}
\gathered
\|\mathfrak{F}^t(u_0)-u_0\|_{C(S^{n-1})}=\Big\|\exp\{\mathcal{A}\int\limits_{0}^t\varphi(V[u(s)])\,ds\}u_0-u_0\Big\|_{C(S^{n-1})}\\
+\int\limits_{0}^t
\Big\|\exp\{\mathcal{A}\int\limits_{s}^t\varphi(V[u(\tau)])\,d\tau\}\Big\|_{C(S^{n-1})}\|F(V[u(s)],u(s))\|_{C(S^{n-1})}\,ds\\
\le(e^{\be \delta}-1)(\|u_0\|_{C(S^{n-1})}+\frac{\|F(V[u_0],u_0)\|_{C(S^{n-1})}+L_1r}{\be}).\qquad\qquad\;\;
\endgathered
\end{equation}
If we decrease $\delta$ such that the inequality
$$
(e^{\be \delta}-1)(\|u_0\|_{C(S^{n-1})}+\frac{\|F(V[u_0],u_0)\|_{C(S^{n-1})}+L_1r}{\be})<\varepsilon
$$
 holds,  then we obtain the required estimate $\|\mathfrak{F}^t(u_0)-u_0\|_{C(S^{n-1})}<\varepsilon$.

Next, we prove property (4): for any $\varepsilon>0$ and $(T,u_0)\in D$
there exists $\delta=\delta(\varepsilon,T,u_0)>0$ such that for all $u_1\in\mathrm{conv}\,(\Bbb R^n)$ $\|u_1-u_0\|_{C(S^{n-1})}<\delta$ the inequality $\|\mathfrak{F}^T(u_1)-\mathfrak{F}^{T}(u_0)\|_{C(S^{n-1})}<\varepsilon$ holds true.

Denote $\widetilde{u}(t)=\mathfrak{F}^t(u_1)$, $u(t)=\mathfrak{F}^t(u_0)$, $0\le t\le T$, $\xi(t)=\|\widetilde{u}(t)-u(t)\|_{C(S^{n-1})}$. Then, from the integral equation \eqref{41}, it follows that
$$
\gathered
\xi(t)\le\Big\|\exp\{\mathcal{A}\int\limits_{0}^t\varphi(V[\widetilde{u}(s)])\,ds\}u_1-
\exp\{\mathcal{A}\int\limits_{0}^t\varphi(V[u(s)])\,ds\}u_0\Big\|_{C(S^{n-1})}\\+\int\limits_{0}^t
\Big\|\exp\{\mathcal{A}\int\limits_{s}^t\varphi(V[\widetilde{u}(\tau)])\,d\tau\}F(V[\widetilde{u}(s)],\widetilde{u}(s))\qquad\qquad\qquad\quad\,
\\ \;\;-
\exp\{\mathcal{A}\int\limits_{s}^t\varphi(V[u(\tau)])\,d\tau\}F(V[u(s)],u(s))\Big\|_{C(S^{n-1})}\,\,ds=
I_1+I_2.
\endgathered
$$
If $\dl<1$, then $\|u_1\|_{C(S^{n-1})}<\|u_0\|_{C(S^{n-1})}+1$. Denote $k=k(u_0,T)=\varphi(V[u_0])+HL\max\limits_{s\in[0,T]}\|u(s)-u_0\|_{C(S^{n-1})}$, $\eta_1=\|A\|HL$. Then for all $t\in[0,T]$ we get the estimate
$$
I_1\le e^{kT\|A\|}(e^{\eta_1\int\limits_{0}^t\xi(s)\,ds}-1)(1+\|u_0\|_{C(S^{n-1})})+e^{\|A\|kT}\|u_1-u_0\|_{C(S^{n-1})}.
$$
Let $K=\sup\limits_{t\in[0,T]}\|F(V[u(t)],u(t))\|_{C(S^{n-1})}$, then
$$
I_2\le\frac{e^{\|A\|kT}L_1}{\eta_1}(e^{\eta_1\int\limits_{0}^t\xi(s)\,ds}-1)+TKe^{\|A\|kT}(e^{\eta_1\int\limits_{0}^t\xi(s)\,ds}-1).
$$
Thus, we get the integral inequality
$$
\xi(t)\le\frac{\alpha}{\eta_1}(e^{\eta_1\int\limits_{0}^t\xi(s)\,ds}-1)+\beta\|u_1-u_0\|_{C(S^{n-1})},\quad 0\le t\le T,
$$
where $\alpha=\alpha(T,u_0)>0$, $\beta=\beta(T,u_0)>0$ are some constants. By using the principle of comparison, it is easy to show that
$$
\xi(t)\le\overline{\xi}(t),\quad 0\le t\le T,
$$
where $\overline{\xi}(t)$ is a  solution of  the IVP
$$
\frac{d\overline{\xi}}{dt}=(\alpha-\beta\|u_1-u_0\|)\overline{\xi}+\overline{\xi}^2,\quad \overline{\xi}(0)=\beta\|u_1-u_0\|,\quad 0\le t\le T.
$$
By applying  the  differential inequality theorem, we obtain
\begin{equation}\label{4-1-2}
\frac{d\overline{\xi}}{dt}\le\alpha\overline{\xi}+\overline{\xi}^2,\quad \overline{\xi}(0)=\beta\|u_1-u_0\|_{C(S^{n-1})},\quad 0\le t\le T.
\end{equation}

 If we integrate the inequality \eqref{4-1-2},  we get the estimate
$$
\overline{\xi}(t)\le\frac{\alpha e^{\alpha t}\overline{\xi}(0)}{\alpha+(1-e^{\alpha t})\overline{\xi}(0)},
$$
 which holds true for all $t\in[0,T]$ for which $\alpha+(1-e^{\alpha t})\overline{\xi}(0)>0$. The number  $\delta$  we can decrease  so much that $\delta<\frac{\alpha}{2\beta(e^{\alpha T}-1)}$, then for all
$t\in[0,T]$ we have $\alpha+(1-e^{\alpha t})\overline{\xi}(0)>\frac{\alpha}{2}$ and as a result we get the inequality
$$
\overline{\xi}(t)\le 2e^{\al t}\overline{\xi}(0),\quad t\in[0,T].
$$
So, if we choose $\dl=\min\{1,\frac{\alpha}{2\beta(e^{\alpha T}-1)},\frac{\varepsilon e^{-\alpha T}}{2\beta}\}$ then for $t=T$ we obtain
$$
\|\mathfrak{F}^T(u_1)-\mathfrak{F}^T(u_0)\|_{C(S^{n-1})}=\|\widetilde{u}(T)-u(T)\|_{C(S^{n-1})}=\xi(T)\le\overline{\xi}(T)<\varepsilon,
$$
and the proof is complete.
\end{proof}

 Next we will  look  for  conditions under which the semiflow $\mathfrak{F}$ is global.
To do this, we establish the conditions of the global existence of the solution $u(t)$ of the integral equation \eqref{41}, using the idea of Lyapunov's direct method \cite{lak}, \cite{lyap}.

First of all, consider the  derivative of functional $V[u]$ along the orbit of semiflow $\mathfrak{F}$
$$
\frac{dV[\mathfrak{F}^t(u)]}{dt}\Big|_{t=0}=\lim\limits_{t\to 0+}\frac{V[\mathfrak{F}^t(u)]-V[u]}{t}.
$$
 We  get  from \eqref{41} the equality
\begin{equation*}
\gathered
V[\mathfrak{F}^t(u_0)]=V\Big[\exp\{\mathcal{A}\int\limits_{0}^t\varphi(V[u(s)])\,ds\}u_0\qquad\;\\ \qquad\qquad\qquad\qquad\qquad+
\int\limits_{0}^t\exp\{\mathcal{A}\int\limits_{s}^t\varphi(V[u(\tau)])\,d\tau\}F(V[u(s)],u(s))\,ds\Big].
\endgathered
\end{equation*}

Applying the  Steiner formula, we obtain
$$
\gathered
V[\mathfrak{F}^t(u_0)]=V\Big[\exp\{\mathcal{A}\int\limits_{0}^t\varphi(V[u(s)])\,ds\}u_0\Big]\\+ntV_1\Big[\exp\{\mathcal{A}\int\limits_{0}^t\varphi(V[u(s)])\,ds\}u_0,
\frac{1}{t}\int\limits_{0}^t\exp\{\mathcal{A}\int\limits_{s}^t\varphi(V[u(\tau)])\,d\tau\}F(V[u(s)],u(s))\,ds\Big]\\+\sum\limits_{k=2}^nt^kC_n^k
V_k\Big[\exp\{\mathcal{A}\int\limits_{0}^t\varphi(V[u(s)])\,ds\}u_0,
\frac{1}{t}\int\limits_{0}^t\exp\{\mathcal{A}\int\limits_{s}^t\varphi(V[u(\tau)])\,d\tau\}F(V[u(s)],u(s))\,ds\Big].
\endgathered
$$
By Liouville--Ostrogradskij theorem, we obtain
$$
V\Big[\exp\{\mathcal{A}\int\limits_{0}^t\varphi(V[u(s)])\,ds\}u_0\Big]=\exp\{\tr A\int\limits_0^t\varphi(V[u(s)])\,ds\}V[u_0].
$$
As a result of continuity of the  mixed volume functional, we get
$$
\frac{dV[\mathfrak{F}^t(u_0)]}{dt}\Big|_{t=0}=\tr A\varphi(V[u_0])V[u_0]+nV_1[u_0,F(V[u_0],u_0)].
$$

For the formulation of conditions  under which the semiflow $\mathfrak{F}$ is global  we need the following assumptions.

\begin{assumption} Assume  that  the  following hold:

(1) there exist the constants $N>0$ and $\al$ such that $\|e^{At}\|\le Ne^{\al t}$, $\,\,t\ge 0$;

(2) there exist the functions $g_*\in C(\Bbb R^+;\Bbb R_+)$ and $g^*\in C(\Bbb R^+;\Bbb R_+)$, satisfying the local Lipschitz condition such that for all $u\in\mathrm{conv}\,(\Bbb R^n)$ the following estimate holds
$$
g_*(V[u])\le nV_1[u,F(V[u],u)]\le g^*(V[u]);
$$

(3)  a maximal solution $\zeta^+(t)$  of the IVP for comparison equation
$$
\frac{d\zeta}{dt}=\tr A\varphi(\zeta)\zeta+g^*(\zeta),\quad\zeta(0)=V[u_0]
$$
and a minimal solution  $\chi_-(t)$ of the IVP for comparison equation
$$
\frac{d\chi}{dt}=\tr A\varphi(\chi)\chi+g_*(\chi),\quad\chi(0)=V[u_0]
$$
 are   infinitely continuable  to the semiaxis $\Bbb R_+$;

(4) there exists the function $F^+\in C(\Bbb R_+;\Bbb R_+)$, $F^+(t,.)$ is a  nondecreasing and such that
$$
\sup\limits_{\chi_-(t)\le s\le\zeta^+(t)}\|F(s,u)\|_{C(S^{n-1})}\le F^+(t,\|u\|_{C(S^{n-1})}, V[u_0])
$$
for all $(t,u)\in\Bbb R_+\times\mathrm{conv}\,(\Bbb R^n)$;

(5) a maximal solution $\omega^+(t)$ of comparison equation
$$
\frac{d\omega}{dt}=\alpha\Lambda^{\pm}(t,V[u_0])\omega+F^+(t,N\omega,V[u_0]),\quad\omega(0)=\|u_0\|_{C(S^{n-1})}
$$
is   infinitely continuable  to the semiaxis  $\Bbb R_+$. Here  we write
 $"+"$ when $\al>0$ and $"-"$ when $\al\le 0$, $\Lambda^+(t,V[u_0])=\max\limits_{\chi_-(t)\le s\le\zeta^+(t)}\varphi(r)$, $\Lambda^-(t,V[u_0])=\min\limits_{\chi_-(t)\le s\le\zeta^+(t)}\varphi(r)$.
\end{assumption}

\begin{theorem}\label{t1} Assume that for a given $u_0\in\mathrm{conv}\,(\Bbb R^n)$ the conditions (1)---(5) of Assumption  4.1 are fulfilled, then for any $t\ge 0$ there exists  $\mathfrak{F}^t(u_0)$.
\end{theorem}
\begin{proof} Let $[0,\Omega^+(u_0))$ be a maximal half-interval of existence of solutions $u(t)=\mathfrak{F}^t(u_0)$ of integral equation \eqref{41}. Assume that $\Omega^+(u_0)<\infty$.  We get from  \eqref{41} and conditions (2), (3) of Assumption 4.1 the norm estimate
$$
\gathered
\|u(t)\|_{C(S^{n-1})}\le N(\exp\{\al\int\limits_{0}^t\Lambda^{\pm}(s)\,ds\}\|u_0\|_{C(S^{n-1})}\qquad
\\ \qquad\qquad\qquad\qquad\qquad+\int\limits_0^t\exp\{\al\int\limits_{s}^t\Lambda^{\pm}(\tau)\,d\tau\}F^+(s,\|u(s)\|_{C(S^{n-1})})\,ds).
\endgathered
$$
 By using the principle of  comparison  it is easy to prove that for all $t\ge 0$, the following inequality  holds
$$
\|u(t)\|_{C(S^{n-1})}\le N\omega^+(t),\quad \omega^+(0)=\|u_0\|_{C(S^{n-1})}.
$$
Let $0<t_1<t_2<\Omega^+(u_0)$, then  the equality
$$
\gathered
u(t_2)=\exp\{\mathcal{A}\int\limits_{t_1}^{t_2}\varphi(V[u(s)])\,ds\}u(t_1)\\ \qquad\qquad\qquad\qquad\quad\;+
\int\limits_{t_1}^{t_2}\exp\{\mathcal{A}\int\limits_{s}^{t_2}\varphi(V[u(\tau)])\,d\tau\}F(V[u(s)],u(s))\,ds
\endgathered
$$
implies the estimate
$$
\gathered
\|u(t_2)-u(t_1)\|_{C(S^{n-1})}\le(\exp\{\al\int\limits_{t_1}^{t_2}\Lambda^{\pm}(s)\,ds\}-1)N\omega^+(t_1)\qquad\qquad\\ \qquad\qquad\qquad\qquad\quad\;+\int\limits_{t_1}^{t_2}
\exp\{\mathcal{A}\int\limits_{s}^{t_2}\al\Lambda^{\pm}(\tau)\,d\tau\}F^+(s,N\omega^+(s))\,ds.
\endgathered
$$
Hence, for $t_1\to\Omega^+(u_0)-0$, $\,t_2\to\Omega^+(u_0)-0$ we get  $\|u(t_2)-u(t_1)\|_{C(S^{n-1})}\to 0$. From the Cauchy criterion it follows that  there exists a limit $$u^*=\lim\limits_{t\to\Omega^+(u_0)-0}u(t)=\lim\limits_{t\to\Omega^+(u_0)-0}\mathfrak{F}^t(u_0).$$

For sufficiently small positive $\tau$ we  get
$$
\mathfrak{F}^{\tau}(u^*)=\mathfrak{F}^{\tau}(\lim\limits_{t\to\Omega^+(u_0)-0}\mathfrak{F}^t(u_0))=\lim\limits_{t\to\Omega^+(u_0)-0}\mathfrak{F}^{\tau}
(\mathfrak{F}^t(u_0))=\mathfrak{F}^{\Omega^+(u_0)+\tau}(u_0).
$$

Thus, the element $\mathfrak{F}^{t}(u_0)$ is defined for $t\in[0,\Omega^+(u_0)+\tau]$, $\tau>0$, which contradicts the definition of $\Omega^+(u_0)$. This contradiction proves that  $\Omega^+(u_0)=+\infty$, completing  the  proof.
\end{proof}

Let $G$  be a closed subset of the space $\mathrm{conv}\,(\Bbb R^n)$. If for any $u_0\in G$  the conditions of Assumption 4.1 are fulfilled, then  the local semiflow $\mathfrak{F}$, narrowed to the set $G$ is global.

\begin{remark}  Consider  a linear control system
\begin{equation}\label{control}
\dot{x}=Ax+u,
\end{equation}
where $x\in \mathbb{R}^n$, $x(0)\in D_0\in \mathrm{conv}\,(\Bbb R^n)$, $u\in U \in \mathrm{conv}\,(\Bbb R^n)$. Let $D(t, D_0)$ be a set of attainability \cite{chernous}  for control system \eqref{control}. Then  a support function $h(t)\in C(S^{n-1})$ for set $D(t, D_0)$ satisfies the  differential equation in Banach space $ C(S^{n-1})$
$$
\frac{dh}{dt}=\mathcal{A}h+h_U,
$$
where $h_U$ is a support function for set $U\in \mathrm{conv}\,(\Bbb R^n)$. So, in this case a family of maps $\mathfrak{F}^t\colon D_0\to D(t,D_0)$  generates a semiflow in space $\mathrm{conv}\,(\Bbb R^n)$, which is a particular case of semiflow  introduced above  if we state $\varphi(s)=1$ and $F(s,u)=h_U$.
\end{remark}

Next we   consider the stability problem in terms of two measures of  semiflow $\mathfrak{F}$.

Assume that $\mathfrak{F}$  is a global semiflow  on the  closed set  $G\subset\mathrm{conv}\,(\Bbb R^n)$, containing a fixed point $\theta_0=\{0\}$ of this semiflow,  i.e. $\mathfrak{F}^t(\theta_0)=\theta_0$, $t\ge 0$.

We shall consider  the stability problem with respect to measures $h_0[u]=h[u]=V[u]$. Such type of stability will be called a divergent stability.

In order to apply the assertion of the general  stability theorem, we  shall formulate some   assumptions.

\begin{assumption} Assume that there exist the functionals $W_k\,\,:\mathrm{conv}\,(\Bbb R^n)\to\Bbb R_+$ which is differentiable along the orbit of semiflow $\mathfrak{F}$ and functions $f_i\in C(\Bbb R^{i+1};\Bbb R)$, $\psi_i(\Bbb R^i_+;\Bbb R)$, $i=\overline{1,k}$ such that

(1) the following inequalities hold

$$
\gathered
nV_1[u,F(V[u],u)]\le\psi_1(V[u])W_1[u],\qquad\qquad\qquad\\
\frac{dW_i[\mathfrak{F}^t(u)]}{dt}\Big|_{t=0}\le f_i(V[u],W_1[u],...,W_i[u])\quad\\+\psi_{i+1}(V[u],W_1[u],...,W_i[u])W_{i+1}[u],\quad i=1,...,k-1,\\
\frac{dW_k[\mathfrak{F}^t(u)]}{dt}\Big|_{t=0}\le f_k(V[u],W_1[u],...,W_k[u]);\quad
\endgathered
$$

(2)  the right sides of  comparison system
\begin{equation}\label{*}
\gathered
\frac{d\xi_0}{dt}=\tr A\varphi(\xi_0)\xi_0+\psi_1(\xi_0)\xi_1,\qquad\qquad\qquad\\
\frac{d\xi_i}{dt}=f_i(\xi_0,...,\xi_i)+\psi_{i+1}(\xi_0,\xi_1,...,\xi_i)\xi_{i+1},\;\\
\frac{d\xi_k}{dt}=f_k(\xi_0,...,\xi_k)\qquad\qquad\qquad\qquad\qquad\;\;\,
\endgathered
\end{equation}
satisfy T. Wazewskij conditions.
\end{assumption}

Assume that comparison system \eqref{*} has a trivial solution $\xi_i=0$.
We define a measure of initial deviations  $h_0[u_0]=\max\limits_{i=\overline{0,\,k}}W_i[u]$.

   The following statement is the immediate consequence of the general stability theorem.

\begin{theorem}\label{t44} Assume that a trivial solution $\xi_i=0$, $\,i=0,...,k$ of comparison system \eqref{*} is

(1) $\xi_0$-stable in cone $\Bbb R_+^{k+1}$;

(2) asymptotically  $\xi_0$-stable in cone  $\Bbb R_+^{k+1}$.

Then the semiflow $\mathfrak{F}$ is

(1) stable in terms of the measures $(h_0[u],V[u])$;

(2)   asymptotically stable in terms of the measures $(h_0[u],V[u])$.
\end{theorem}

 Assume that the trivial solution   $\xi_i=0$, $\,i=0,...,k$ of comparison system  \eqref{*} is  $u_0$-unstable. So if in (1) of  Assumption 4.2 we take the inequalities with a reversed sign then the semiflow $\mathfrak{F}$  would be unstable in terms of two measures $(h_0[u],V[u])$.

Note that in the important special case when $A =0$  the semiflow $\mathfrak{F}$, as a rule, can not be stable (for example, when it is generated by SDEs).

 In this case, for the qualitative analysis of the dynamics of the semiflow it is appropriate  to use the notion of practical stability on a finite  time interval.
  Next we present a stability theorem  for this property. In this case, there is no need to assume that $\theta_0$ is a fixed point of the semiflow.

\begin{theorem}\label{t45}  Assume that the conditions of Assumption 4.2 are fulfilled  and the following  inequality holds
$$
\xi_0(T;\lm,...,\lm)<A.
$$

Then the semiflow  $\mathfrak{F}$ is practical $(\lm,A,T)$- stable in terms of two measures $(h_0[u],V[u])$ on a finite  time interval.

If in (1) of Assumption 4.2   we take the inequalities with a reversed sign  and assume that the following  inequality holds
$$
\xi_0(T;\lm,...,\lm)\ge A,
$$
then the semiflow  $\mathfrak{F}$ is  practical$(\lm,A,T)$- unstable in terms of two measures $(h_0[u],V[u])$ on a finite  time interval.
\end{theorem}

The estimates obtained in the proof of the conditions for global semiflow $\mathfrak{F}$ allow us also to formulate some sufficient conditions for the stability of the semiflow $\mathfrak{F}$ with respect to measures that are Hausdorff metric, i.e. $h_0[u]=h[u]=\|u\|_{C(S^{n-1})}$.

\begin{theorem}\label{t3} Assume that the conditions (1)--(4)  of the Assumption 4.1 are fulfilled, and there exist
$$
\hat{\Lambda}^{+}(t)=\sup\limits_{u_0\in B_r(0)}\Lambda^{+}(t,V[u_0]),\quad \hat{\Lambda}^{-}(t)=\inf\limits_{u_0\in B_r(0)}\Lambda^{+}(t,V[u_0]), \quad
$$$$
\hat{F}^+(t,\omega)=\sup\limits_{u_0\in B_r(0)}F^+(t,\omega,V[u_0]).
$$

If    the solution $\omega=0$ of comparison equation
$$
\frac{d\omega}{dt}=\alpha\hat{\Lambda}^{\pm}(t)\omega+\hat{F}^+(t,N\omega),\quad\omega(0)=\|u_0\|
$$
 is

(1) stable by Lyapunov;

(2)  asymptotically stable by Lyapunov.

Then  the semiflow $\mathfrak{F}$ is

(1) stable in terms  of measures $h_0[u]=h[u]=\|u\|_{C(S^{n-1})}$;

(2) asymptotically stable in terms  of measures $h_0[u]=h[u]=\|u\|_{C(S^{n-1})}$.
\end{theorem}

 Consider the stability problem  by a linear approximation for an arbitrary fixed point $u=u^*$ of the semiflow $\mathfrak{F}$, i.e. the stability of the  semiflow in terms of the  measures $h[u_0]=h[u]=\|u-u^*\|_{C(S^{n-1})}=d_{H}(u,u^*)$.

 For this purpose, it is necessary to introduce additional assumptions about the differentiability of the mapping $\,F\,:\Bbb R_+\times\mathrm{conv}\,(\Bbb R^n)\to\mathrm{conv}\, (\Bbb R^n)$ and  function $\varphi\,\,:\Bbb R_+\to\Bbb R_+$:

(a) for any $(s_0,u_0)\in\Bbb R_+\times\mathrm{conv}\,(\Bbb R^n)$ there exist the element $F_s(s_0,u_0)\in C(S^{n-1})$  and a linear continuous operator $F_u(s_0,u_0)\in\mathfrak{L}(C(S^{n-1}))$  such that for all  $s$ from a certain neighborhood of $s_0$ and $u\in\mathrm{conv}\,(\Bbb R^n)$ from a certain neighborhood (in space $\mathrm{conv}\,(\Bbb R^n)$) of point $u_0$  we have the representation
$$
F(s,u)-F(s_0,u_0)=F_s(s_0,u_0)(s-s_0)+F_u(s_0,u_0)(u-u_0)+o(\varrho),
$$
where $\varrho=\sqrt{|s-s_0|^2+\|u-u_0\|_{C(S^{n-1})}^2}$.

(b)  there exists a continuous derivative $\frac{d\varphi}{dV}$ of function $\varphi(V)$.

 Let us define the variables of a perturbed motion  $\dl u(t)=\mathfrak{F}^t(u)-u^*$, $\dl V(t)=V[\mathfrak{F}^t(u)]-V[u^*]$,
then the equation of  perturbed motion be of the form
$$
\frac{d\delta V}{dt}=\tr A\varphi(V[u^*]+\dl V)(V[u^*]+\dl V)+nV_1[u^*+\dl u,F(V[u^*]+\dl V,u^*+\dl u)].
$$

Taking into account the assumption of differentiability of the mapping $F$ and using the formula (9) from ( \cite{alex1}, p. 969), we obtain
$$
\gathered
\frac{d\delta V}{dt}=(\tr A(\varphi(V[u^*])+V[u^*]\frac{d\varphi}{dV}(V[u^*]))+nV_1[u^*,F_V(V[u^*],u^*)])\dl V\\+n(n-1)(V[u^*,...,u^*,F(V[u^*],u^*),\dl u]+V_1[u^*,F_u(V[u^*],u^*)\dl u)])+R_1.
\endgathered
$$
Here $R_1=o(\,\|\dl u\|_{C(S^{n-1})}+|\dl V|\,)$ then $ \|\dl u\|_{C(S^{n-1})}+|\dl V| \to 0$.

Denote
$$
\gm_0=\tr A(\varphi(V[u^*])+V[u^*]\frac{d\varphi}{dV}(V[u^*]))+nV_1[u^*,F_V(V[u^*],u^*)],
$$
then by Cauchy formula, we get
$$
\gathered
\dl V(t)=e^{\gamma_0 t}\dl V(0)+\int\limits_{0}^te^{\gm_0(t-s)}n(n-1)(V[u^*,...,u^*,F(V[u^*],u^*),\dl u(s)]\\ \quad\;\,+V_1[u^*,F_u(V[u^*],u^*)\dl u(s))])+R_1)\,ds
\endgathered
$$
Hence,  from  the integral representation \eqref{2111}  follows the estimate
$$
|\dl V(t)|\le e^{\gamma_0 t}|\dl V(0)|+\int\limits_{0}^te^{\gm_0(t-s)}\Delta_0\|\dl u(s)\|_{C(S^{n-1})}\,ds+o(\,\|\dl u\|_{C(S^{n-1})}+|\dl V|\,),
$$
where
$$
\gathered
\Delta_0=(n-1)\int\limits_{S^{n-1}}\big(F[u^*,...,u^*,F(V[u^*],u^*));d\omega]\\\;\;\;+\|F_u(V[u^*],u^*)\|_{\mathfrak{L}(C(S^{n-1}))}F[u^*,...,u^*;d\omega]\big).
\endgathered
$$

 Consider now the equation for  variation $\dl u$:
$$
\gathered
\frac{d\dl u}{dt}=\varphi(V[u^*]+\dl V)\mathcal{A}u-\varphi(V[u^*])\mathcal{A}u^*+F_u(V[u^*],u^*)\delta u\\+F_V(V[u^*],u^*)\dl V(t)+R_2,\qquad\qquad\qquad\qquad\quad\;\;
\endgathered
$$
where $R_2=o(\|\dl u\|_{C(S^{n-1})}+|\dl V|)$ then $ \|\dl u\|_{C(S^{n-1})}+|\dl V| \to 0$.

 It is easy to see that this equation
is equivalent to the integral equation
$$
\gathered
\dl u(t)=\exp\Big\{\mathcal{A}\int\limits_0^t\varphi(V[\mathfrak{F}^s(u)])\,ds\Big\}\dl u(0)\\ \qquad\qquad\qquad\;\;+\int\limits_0^t\exp\Big\{\mathcal{A}\int\limits_s^t\varphi(V[\mathfrak{F}^{\tau}(u)])\,d\tau\Big\}\dl\varphi(s)\,ds\mathcal{A}u^*
\endgathered
$$

$$
\gathered
+
\int\limits_0^t\exp\Big\{\mathcal{A}\int\limits_s^t\varphi(V[\mathfrak{F}^{\tau}(u)])\,d\tau\Big\}F_u(V[u^*],u^*)\dl u(s)\,ds\;\\+\int\limits_0^t\exp\Big\{\mathcal{A}\int\limits_s^t\varphi(V[\mathfrak{F}^{\tau}(u)])\,d\tau\Big\}F_V(V[u^*],u^*)\dl V(s)\,ds\\+\int\limits_0^t\exp\Big\{\mathcal{A}\int\limits_s^t\varphi(V[\mathfrak{F}^{\tau}(u)])\,d\tau\Big\}R_2\,ds,\qquad\qquad\qquad\quad
\endgathered
$$
where $\dl \varphi(s)=\varphi(V[u^*]+\dl V(s))-\varphi(V[u^*])=\frac{d\varphi}{dV}(V[u^*])\dl V+o(\dl V)$.

We get from the norm estimation  the following integral inequality
$$
\gathered
\|\dl u(t)\|_{C(S^{n-1})}\le N(\exp\Big\{\al\varphi(V[u^*])t+\int\limits_0^t|\dl \varphi(s)|\,ds\Big\}\|\dl u(0)\|_{C(S^{n-1})}\\+\|A\|\,\|u^*\|_{C(S^{n-1})}\int\limits_0^t\exp\Big\{\al\varphi(V[u^*])(t-s)+\int\limits_s^t|\dl \varphi(\tau)|\,d\tau\Big\}|\delta \varphi(s)|\,ds\\+
\int\limits_0^t\exp\Big\{\al\varphi(V[u^*])(t-s)+\int\limits_s^t|\dl \varphi(\tau)|\,d\tau\Big\}\|F_u(V[u^*],u^*)\|_{\mathfrak{L}(C(S^{n-1}))}\|\dl u(s)\|_{C(S^{n-1})}\,ds\\+\int\limits_0^t\exp\Big\{\alpha\varphi(V[u^*])(t-s)+\int\limits_s^t|\dl \varphi(\tau)|\,d\tau\Big\}\|F_V(V[u^*],u^*)\|_{C(S^{n-1})}|\dl V(s)|\,ds\\+
\int\limits_0^t\exp\Big\{\alpha\varphi(V[u^*])(t-s)+\int\limits_s^t|\dl \varphi(\tau)|\,d\tau\Big\}\|R_2\|_{C(S^{n-1})}\,ds).
\endgathered
$$

By using the comparison theorem, it is easy to show that
$$
\|\dl u(t)\|_{C(S^{n-1})}\le N\omega_1(t),\quad |\dl V(t)|\le\omega_2(t),\quad t\ge 0,
$$
where $\omega_i(t)$, $i=1,2$ is a solution of IVP for a system of differential equations
$$
\gathered
\frac{d\omega_1}{dt}=(\al\varphi(V[u^*])+N\|F_u(V[u^*],u^*)\|_{\mathfrak{L}(C(S^{n-1}))})\omega_1\\
+\Big(\|A\|\Big|\frac{d\varphi}{dV}(V[u^*])\Big|+\|F_V(V[u^*],u^*)\|_{C(S^{n-1})}\Big)\omega_2\\+
\|A\|N\Big|\frac{d\varphi}{dV}(V[u^*])\Big|\omega_1\omega_2+o(\omega),\quad\omega_1(0)=\|\dl u(0)\|_{C(S^{n-1})},\\
\frac{d\omega_2}{dt}=N\Delta_0\omega_1+\gm_0\omega_2+o(\omega),\quad\omega_2(0)=|\delta V(0)|.
\endgathered
$$

The above considerations allow us to prove the following statement.

\begin{theorem}\label{t47} Assume that the mapping $F\,:\Bbb R_+\times \mathrm{conv}\,(\Bbb R^n)\to\mathrm{conv}\,(\Bbb R^n)$ is
continuously differentiable in the neighborhood of $(V[u^*],u^*)\in\Bbb R_+\times \mathrm{conv}\,(\Bbb R^n)$, and a function $\varphi(V)$ is continuously differentiable in the neighborhood of  the point $V=V[u^*]$ and the following inequalities hold
$$
\gathered
\gm_0<0,\\ (\al\varphi(V[u^*])+N\|F_u(V[u^*],u^*)\|_{\mathfrak{L}(C(S^{n-1}))})\gm_0
\\+N\Delta_0\Big(\|A\|\Big|\frac{d\varphi}{dV}(V[u^*])\Big|+\|F_V(V[u^*],u^*)\|_{C(S^{n-1})}\Big)>0.
\endgathered
$$
Here
$$
\gathered
\gm_0=\tr A(\varphi(V[u^*]+V[u^*]\frac{d\varphi}{dV}(V[u^*]))+nV_1[u^*,F_V(V[u^*],u^*)].
\endgathered
$$

Then the fixed point $u=u^*$ of the semiflow $\mathfrak{F}$  is asymptotically  stable by Lyapunov.
\end{theorem}
\textbf{Proof.} From the conditions  of the theorem it follows that the equilibrium state $\omega_1=\omega_2=0$  of comparison system is  asymptotically  stable. This means that for any $\varepsilon>0$ there exists a positive number $\dl_1=\dl_1(\varepsilon)$
such that the inequalities $\omega_1(0)<\dl_1$, $\omega_2(0)<\dl_1$ imply the estimate $\omega_1(t)<\varepsilon$ for all  $t\ge 0$. Let $L$ be a Lipschitz constant for the functional $V[u]$ relative to the ball $K_1(0)\subset\mathrm{conv}\,(\Bbb R^n)$ and choose $\dl_0=\min\{1,\dl_1(1/N)/L,\dl_1(1/N)/N\}$.

 Then the inequality $\|u-u^*\|_{C(S^{n-1})}<\dl_0$ implies the estimate $|V[u]-V[u^*]|\le L\|u-u^*\|_{C(S^{n-1})}<\dl_1$, therefore $\|\mathfrak{F}^t(u)-u^*\|_{C(S^{n-1})}<N\omega_1(t)<1$ for all $t\ge 0$. Given $\,\varepsilon>0$, we choose $\dl(\varepsilon)=\min\{\dl_0,\frac{\dl_1(\varepsilon/N)}{L},\dl_1(\varepsilon/N)\}$. Then $\omega_1(0)=\|\dl u(0)\|_{C(S^{n-1})}<\dl_1(\varepsilon/N)$, $\omega_2(0)=|\dl V(0)|<L\|u-u^*\|_{C(S^{n-1})}<\dl_1(\varepsilon/N)$ and, as a consequence, $\|\mathfrak{F}^t(u)-u^*\|_{C(S^{n-1})}<N\omega_1(t)<\varepsilon$ for all $t\ge 0$, which proves the stability of the fixed point $u=u^*$  of the semiflow  $\mathfrak{F}$. Asymptotic stability is proved similarly. The  proof  is  complete.

\section{Examples}%5
Consider an examples of application of  obtained results to the study of  specific semiflows in  the space of convex compacts  $\mathrm{conv}\,(\Bbb R^n)$.

{\it Example 5.1.} Consider a semiflow whose parameters are of the form $A=-I$, $F(V,u)=\psi(V)K$, $K$ is a unit ball in space $\Bbb R^n$, $\varphi\in C^1(\Bbb R_+;\Bbb R_+)$, $\psi\in C^1(\Bbb R_+;\Bbb R_+)$. The fixed points $u=u^*\in\mathrm{conv}\,(\Bbb R^n)$  of this semiflow  are defined as the solutions of the equation
$$
-\varphi(V[u^*])u^*+\psi(V[u^*])K=0.
$$
Let $\lambda_0>0$ be  a root of the equation
$$
\lambda_0=\Big[\frac{\psi(\lambda_0)}{\varphi(\lambda_0)}\Big]^n\frac{\pi^{n/2}}{\Gamma(1+n/2)},
$$
where $\Gamma(x)$  is a  gamma function of Euler, then $u^*=\frac{\psi(\lambda_0)}{\varphi(\lambda_0)}K$.

To study the stability of a fixed point $u=u^*$ of the semiflow $\mathfrak{F}$ in terms of Hausdorff measure ($h_0[u]=h[u]=d_H(u,u^*)$)  we  shall apply the Theorem  \ref{t47}. So, we have
$$
\gathered
\gamma_0=-n(\varphi(\lambda_0)+\lambda_0\varphi^{\prime}(\lambda_0))+nV_1\Big[\frac{\psi(\lambda_0)}{\varphi(\lambda_0)}K,\psi^{\prime}(\lambda_0)K\Big]
\;\;\\=-n\Big[\varphi(\lambda_0)+\lambda_0\psi(\lambda_0)\frac{d}{d\lambda}\Big[\frac{\varphi(\lambda)}{\psi(\lambda)}\Big]_{\lambda=\lambda_0}\Big].\qquad\qquad\;
\endgathered
$$
Therefore, it is obvious that the conditions of Theorem \ref{t47}  are reduced to the single inequality
$$
\frac{d}{d\lambda}\Big[\ln\Big(\lambda\frac{\varphi(\lambda)}{\psi(\lambda)}\Big)\Big]_{\lambda=\lambda_0}>0.
$$
{\it Example 5.2.} Let $n=2$, $A=-I$, $F(s,u)=\psi(s)Bu$, $B\in\mathfrak{L}(\Bbb R^2)$, $B^2=0$.

To investigate the stability of the semiflow $\mathfrak{F}$ in terms of two measures $(h_0[u],S[u])$, we introduce two auxiliary functionals $W_0[u]=S[u]$, $W_1[u]=S[u,Bu]$, and calculate the total derivatives of these functionals along the  orbit of semiflow $\mathfrak{F}$:
\begin{equation}\label{e51}
\gathered
\frac{dW_0[\mathfrak{F}^t(u)]}{dt}\Big|_{t=0}=-2\varphi(W_0[u])W_0[u]+2\psi(W_0[u])W_1[u],\\
\frac{dW_1[\mathfrak{F}^t(u)]}{dt}\Big|_{t=0}=-2\varphi(W_0[u])W_1[u].\qquad\qquad\qquad\qquad
\endgathered
\end{equation}

The comparison system in this case is of  the form
\begin{equation}\label{e52}
\gathered
\frac{d\xi_0}{dt}=-2\varphi(\xi_0)\xi_0+2\psi(\xi_0)\xi_1,\;\;\;\\
\frac{d\xi_1}{dt}=-2\varphi(\xi_0)\xi_1.\qquad\qquad\qquad
\endgathered
\end{equation}

It should be noted that since the relations \eqref{e51} for total derivatives of the functionals $ W_0[u]$ and $W_1[u]$ along the  orbit of semiflow $\mathfrak{F}$ are of the form of equalities,   there is no need to require  for comparison system \eqref{e52} the fulfillment of the  Wazewskij condition.

 We shall next study   the $\xi_0$ -stability of solution $\xi_0=\xi_1=0$ for system \eqref{e52}  using the Lyapunov function
$$
V(\xi_0,\xi_1)=\frac{1}{2}(\xi_0^2+\beta\xi_1^2).
$$
The total derivative of this function along the solutions of the comparison  system is of the form
$$
\frac{dV}{dt}\Big|_{\eqref{e52}}=-2\varphi(\xi_0)\xi_0^2+2\psi(\xi_0)\xi_1\xi_0-2\beta\varphi(\xi_0)\xi_1^2.
$$
 The sufficient conditions for negative definiteness of this derivative are of the form
$$
0<\inf\limits_{s>0}\varphi(s)\le\sup\limits_{s>0}\varphi(s)<\infty,\quad \sup\limits_{s>0}\frac{\psi(s)}{\varphi(s)}<\infty.
$$

 In this case, by  Lyapunov theorem the solution $\xi_0=\xi_1=0$ of  system \eqref{e52} is asymptotically stable.

If  the function $\varphi(s)$ is nonincreasing,  then the right sides of  system  \eqref{e52} satisfy   Wazewskij   condition and  stability  conditions of solution  $\xi_0=\xi_1=0$  can be obtained using the    Martynyuk–-Obolenskij  criterion \cite{mart-ob}--\cite{ob-1} for autonomous  Wazewskij systems.   This criterion allows us to state that if the inequalities  $\varphi(s)>0$, $\psi(s)>0$ are fulfilled for  $s>0$ and $\varphi(s)$ is nonincreasing,  then the solution  $\xi_0=\xi_1=0$ of system \eqref{e52}  is asymptotically stable.

Theorem \ref{t31}  implies that  if  the solution $\xi_0=\xi_1=0$  of comparison system  is asymptotically stable, then  the semiflow $\mathfrak{F}$  is stable in terms of measures $(h_0[u],S[u])$, $h_0[u]=\max\{S[u],S[u,Bu]\}$.

The geometric meaning of the  measure of initial deviations $h_0[u]$ is that it restricts not only the area of convex compact $u\in\mathrm{conv}\,(\Bbb R^2)$, but also the value of $S[u, Bu]$, which is a   projection of this compact on the line $\mathcal R(B)$ along a line $\ker(B)$.

{\it Example 5.3.} Let $n=2$, $A=-I$, $F(s,u)=\psi(s)Bu$, $B\in\mathfrak{L}(\Bbb R^2)$, $B^k=I$, $k$ is a natural number.

To investigate the stability of the semiflow $\mathfrak{F}$ in terms of two measures $(h_0[u],S[u])$, we introduce  auxiliary functionals $W_i[u]=S[u,B^iu]$. The total derivatives of these functionals along the orbit of semiflow $\mathfrak{F}$ are of the form

\begin{equation}\label{e53}
\gathered
\frac{dW_0[\mathfrak{F}^t(u)]}{dt}\Big|_{t=0}=-2\varphi(W_0[u])W_0[u]+2\psi(W_0[u])W_1[u],\qquad\qquad\quad\;\;\\
\quad\frac{dW_i[\mathfrak{F}^t(u)]}{dt}\Big|_{t=0}=-2\varphi(W_0[u])W_0[u]+2\psi(W_0[u])(W_{i-1}[u]+W_{i+1}[u]),\\
\frac{dW_{k-1}[\mathfrak{F}^t(u)]}{dt}\Big|_{t=0}=-2\varphi(W_0[u])W_{k-1}[u]+2\psi(W_0[u])(W_{k-2}[u]+W_{0}[u]).
\endgathered
\end{equation}
The comparison system is of the form
\begin{eqnarray}\label{e54}
 \frac{d\xi_0}{dt} &=& -2\varphi(\xi_0)\xi_0+2\psi(\xi_0)\xi_1,\nonumber \\
  \frac{d\xi_i}{dt} &=& -2\varphi(\xi_0)\xi_i+\psi(\xi_0)(\xi_{i-1}+\xi_{i+1}),\quad i=1,2,...,k-2, \\
  \frac{d\xi_{k-1}}{dt} &=& -2\varphi(\xi_0)\xi_{k-1}+\psi(\xi_0)(\xi_{k-2}+\xi_{0}). \nonumber
\end{eqnarray}
%\begin{equation}\label{e54}
%\gathered
%\frac{d\xi_0}{dt}=-2\varphi(\xi_0)\xi_0+2\psi(\xi_0)\xi_1,\\
%\frac{d\xi_i}{dt}=-2\varphi(\xi_0)\xi_i+\psi(\xi_0)(\xi_{i-1}+\xi_{i+1}),\quad i=1,2,...,k-2,\\
%\frac{d\xi_{k-1}}{dt}=-2\varphi(\xi_0)\xi_{k-1}+\psi(\xi_0)(\xi_{k-2}+\xi_{0}).
%\endgathered
%\end{equation}

It should be noted that since the relations \eqref{e53}  are  equalities,   there is no need to require  for comparison system \eqref{e54} the fulfillment of the  Wazewskij condition.

Theorem \ref{t41} implies that if solution $\xi=0$  of comparison system  \eqref{e54} is $\xi_0$-stable (asymptotically stable), then  the semiflow is stable (asymptotically stable) in terms of two measures   $(h_0[u],S[u])$, $h_0[u]=\max\limits_{i=\overline{0,k-1}}S[u,B^iu]$.

Consider the special case for $k=2$. Choose for comparison system
\begin{equation}\label{e55}
\gathered
\frac{d\xi_0}{dt}=-2\varphi(\xi_0)\xi_0+2\psi(\xi_0)\xi_1,\\
\frac{d\xi_1}{dt}=2\psi(\xi_0)\xi_1-2\varphi(\xi_0)\xi_1\;\;\;\;
\endgathered
\end{equation}
 the Lyapunov  function  $V(\xi_0,\xi_1)=\frac{1}{2}(\xi_0^2+\xi_1^2)$. Then the total derivative of this function along the solutions of system \eqref{e55} is of the form
$$
\frac{dV}{dt}\Big|_{\eqref{e55}}=2(-\varphi(\xi_0)\xi_0^2+2\psi(\xi_0)\xi_0\xi_1-\varphi(\xi_0)\xi_1^2).
$$
Sufficient conditions for negative definiteness of the total derivative are of the form
$$
0<\inf\limits_{s>0}\varphi(s)\le\sup\limits_{s>0}\varphi(s)<\infty,\quad \sup\limits_{s>0}\frac{\psi(s)}{\varphi(s)}<1.
$$
If  the function $\varphi(s)$ is nonincreasing    and function   $\psi(s)$ is nondecreasing, then the comparison system  \eqref{e55} satisfies  Wazewskij condition  and   we can applied the   Martynyuk–-Obolenskij criterion  for stability investigation.  Then the conditions
$$
\gathered
(\forall s>0)\quad (\varphi(s)>0,\quad \psi(s)>0),\\
(\exists s_0>0)\,\quad (\frac{\psi(s_0)}{\varphi(s_0)}<1)
\endgathered
$$
guarantees the asymptotic stability of solutions $\xi_0=\xi_1=0$ of comparison system  \eqref{e55}.

If $\varphi(s)=1$ $\,\psi(s)=1/2$, then  we can  integrate the comparison equation and so we get a formula for   area
$$
S[\mathfrak{F}^t(u)]=\frac{1}{2}e^{-t}(S[u]+S[Bu,u])+\frac{1}{2}e^{-3t}(S[u]-S[u,Bu]).
$$
It is clear that in this case the semiflow $\mathfrak{F}$ is asymptotically stable in terms of two measures  $(h_0[u],S[u])$, $h_0[u]=\max\{S[u],S[u,Bu]\}$.

Consider the case then $k=3$.  The comparison system is of the form
\begin{equation}\label{e56}
\gathered
\frac{d\xi_0}{dt}=-2\varphi(\xi_0)\xi_0+2\psi(\xi_0)\xi_1,\qquad\quad\\
\frac{d\xi_1}{dt}=-2\varphi(\xi_0)\xi_1+2\psi(\xi_0)(\xi_0+\xi_2),\,\\
\frac{d\xi_2}{dt}=-2\varphi(\xi_0)+2\psi(\xi_0)(\xi_1+\xi_0).\quad
\endgathered
\end{equation}
Consider the auxiliary function $V(\xi_0,\xi_1,\xi_2)=\frac{1}{2}(\xi_0^2+\xi_1^2+\xi_2^2)$. The total derivative of this function along the solutions of system \eqref{e56} is of the form
$$
\frac{dV}{dt}\Big|_{\eqref{e56}}=-2\varphi(\xi_0)(\xi_0^2+\xi_1^2+\xi_2^2)+\psi(\xi_0)(3\xi_0\xi_1+2\xi_1\xi_2+\xi_0\xi_2).
$$
The total derivative $\frac{dV}{dt}\Big|_{\eqref{e56}}$ is negative definite if
$$
0<\inf\limits_{s>0}\varphi(s)\le\sup\limits_{s>0}\varphi(s)<\infty,\quad \sup\limits_{s>0}\frac{\psi(s)}{\varphi(s)}<\lambda^*,
$$
where $\lambda^*$ is a least positive root of the cubic equation $3\lambda^3+14\lambda^2-16=0$.

If  the function $\varphi(s)$ is nonincreasing    and  the function   $\psi(s)$ is nondecreasing, then the comparison system  \eqref{e56} satisfies  Wazewskij condition  and   we can applied the   Martynyuk–-Obolenskij criterion \cite{mart-ob}  for stability investigation.  Then the conditions of asymptotic stability for solution  $\xi_0=\xi_1=0$  of comparison system  \eqref{e56} are of the form
 $$
\gathered
(\forall s>0)\quad (\varphi(s)>0,\quad \psi(s)>0),\\
(\exists s_0>0)\,\quad (\frac{\psi(s_0)}{\varphi(s_0)}<1).
\endgathered
$$

 It should be noted also that if $\varphi\in C^1$ and $\psi\in C^1$, then  by applying  the Routh--Hurwitz  conditions  we get the inequality
  $\frac{\psi(0)}{\varphi(0)}<1$. However, in this case, we can only guarantee a local stability of solutions $\xi_0=\xi_1=0$ of comparison system, in contrast to the conditions obtained using the Lyapunov function that guarantee  a global stability.
 We can say that  the stronger restrictions on the initial conditions than on the current deviations are significant and  related not only with the research method, but  with the essence of the problem too. To show this let us consider the following case.

Consider the special case $k=4$, $B=\mathcal{J}=\begin{pmatrix}
                                        0 & -1 \\
                                        1& 0 \\
                                      \end{pmatrix}$.  The comparison system  is of the form
\begin{equation}\label{e55*}
\gathered
\frac{d\xi_0}{dt}=-2\varphi(\xi_0)\xi_0+2\psi(\xi_0)\xi_1,\;\;\;\;\;\\
\frac{d\xi_1}{dt}=2\psi(\xi_0)\xi_1+\varphi(\xi_0)(\xi_0+\xi_2),\\
\frac{d\xi_2}{dt}=2\psi(\xi_0)\xi_2+\varphi(\xi_0)(\xi_1+\xi_3),\\
\frac{d\xi_3}{dt}=2\psi(\xi_0)\xi_3+\varphi(\xi_0)(\xi_0+\xi_2).
\endgathered
\end{equation}
If $\varphi(s)=1$ $\,\psi(s)=1/2$  then the comparison system  \eqref{e55*} is integrated and so we can obtain a formula for area
$$
\gathered
S[\mathfrak{F}^t(u)]=\frac{1}{8}e^{-t}(2S[u]+3S[Bu,u]+2S[B^2u,u]+S[B^3u,u])\\
\qquad\qquad+\frac{1}{8}e^{-3t}(2S[u]-3S[Bu,u]+2S[B^2u,u]-S[B^3u,u])\\
\qquad\qquad\qquad\,+\frac{1}{2}e^{-2t}(S[u]-S[B^2u,u])+\frac{1}{4}e^{-2t}t(S[Bu,u]-S[B^3u,u]).
\endgathered
$$
It is clear that the semiflow $\mathfrak{F}$  is asymptotically stable in terms of measures $(h_0[u],S[u])$, $h_0[u]=\max\limits_{i=0,1,2,3}[u,B^iu]$.
\newline Next  we show that the considered semiflow is unstable in terms of measures $(S[u],S[u])$.

 Consider the sequence of convex compacts
$$
u_N=\{(x,y)\in\Bbb R^2\,\quad|\,\,|x|\le\frac{N}{2},\,\,y=0\}\in\mathrm{conv}\,(\Bbb R^2),
$$
then
$$
Bu_N=\{(x,y)\in\Bbb R^2\,\quad|\,\,|y|\le\frac{N}{2},\,\,x=0\}\in\mathrm{conv}\,(\Bbb R^2).
$$
As a result of centrally symmetric of sets $u_N$  we obtain
$$
B^2u_N=u_N,\,B^3u_N=Bu_N,
$$
It is obvious that $S[u_N]=0$, $S[u_N,Bu_N]=N^2/2$, $S[u_N,B^2u_N]=0$, $S[u_N,B^3u_N]=N^2/2$.  Therefore
$$
S[\mathfrak{F}^t(u_N)]=\frac{1}{4}(e^{-t}-e^{-3t})N^2.
$$
For a fixed $t> 0$  the area $S[\mathfrak{F}^t(u_N)]\to\infty$ for $N\to\infty$,  and this proves that the semiflow  $\mathfrak{F}$ is unstable in terms of measures  $(S[u],S[u])$.

%We shall next  establish that  the considered semiflow is unstable in terms of  measures  $(S[u],S[u])$.
%
%Consider a sequence of convex compacts
%$$
%u_N=\{(x,y)\in\Bbb R^2\,\big|\,|x|\le\frac{N}{2},\,\,y=0\}\in\mathrm{conv}\,(\Bbb R^2),
%$$
%then
%$$
%Bu_N=\{(x,y)\in\Bbb R^2\,\,\big|\,|y|\le\frac{N}{2},\,\,x=0\}\in\mathrm{conv}\,(\Bbb R^2).
%$$
%It is obvious that $S[u_N]=0$, $S[u_N,Bu_N]=N^2/2$. Therefore
%$$
%S[\mathfrak{F}^t(u_N)]=\frac{1}{4}(e^{-t}-e^{-3t})N^2.
%$$
%For a fixed $t> 0$  the area $S[\mathfrak{F}^t(u_N)]\to\infty$ for $N\to\infty$,  and this proves that the semiflow  $\mathfrak{F}$ is unstable in terms of measures  $(S[u],S[u])$.

 We note that the need to consider a various measures for initial and current deviations is known and is also typical for systems of partial differential equations, and is associated with the existence of non-equivalent norms in infinite-dimensional space
 \cite{sirazet}--\cite{knopswilkes}.

{\it Example 5.4.} Consider  a set differential equations
\begin{equation}\label{e57}
D_Hu(t)=Bu(t),
\end{equation}
where $u(t)\in\mathrm{conv}\,(\Bbb R^2)$, $B\in\mathfrak{L}(\Bbb R^2)$, $\det B<0$, $\tr B\geq0$.

Since $\diam\mathfrak{F}^t(u)$ is a nondecreasing as a function of $t>0$, then there is no sense to consider the stability or  asymptotic stability of the semiflow $ \mathfrak{F}$, generated by the equation \eqref{e57}.

 It is expedient to consider the practical stability on a finite interval, or $(\lambda,A,T)$--stability in terms of two measures $(h_0,h)$, $\,h_0[u]=\max\{S[u],S[u,Bu]\}$.

Consider the functionals $W_0[u]=S[u]$, $W_1[u]=S[u,Bu]$. Then
$$
\gathered
\frac{dW_0[\mathfrak{F}^t(u)]}{dt}\Big|_{t=0}=2W_1[u],\qquad\qquad\qquad\qquad\\
\frac{dW_1[\mathfrak{F}^t(u)]}{dt}\Big|_{t=0}=|\det B|W_0[u]+S[u,B^2u].
\endgathered
$$
By the theorem of Cayley--Hamilton   we get $B^2=(\tr B)B-(\det B)I$, therefore
$$
S[u,B^2u]=S[u,((\tr B)B-(\det B)I)u]\le\tr B W_1[u]+|\det B|W_0[u].
$$
 So the comparison system is of the form
\begin{equation}\label{e58}
\gathered
\frac{d\xi_0}{dt}=2\xi_1,\qquad\qquad\qquad\quad\\
\frac{d\xi_1}{dt}=2|\det B|\xi_0+\tr B\xi_1.
\endgathered
\end{equation}
The right-hand sides of the system \eqref{e58} satisfy the conditions of Wazewskij. By integrating of this comparison system we obtain the estimate
$$
\gathered
S[\mathfrak{F}^t(u)]\le\frac{1}{\sqrt{\tr^2 B+16|\det B|}}((2S[u,Bu]-\mu_-S[u])e^{\mu_+t}\;\;\\+(\mu_+S[u]-2S[u,Bu])e^{\mu_-t}),\quad t\ge 0,\qquad
\endgathered
$$
where
$$
\mu_{\pm}=\frac{4|\det B|\pm\sqrt{\tr^2 B+16|\det B|}}{2}.
$$
 The conditions of  practical $(\lm,A,T)$--stability are of the form
$$
2+(2-\mu_-)e^{\mu_+T}<\frac{A\sqrt{\tr^2 B+16|\det B|}}{\lambda}.
$$

{\it Example 5.5.}  Let us establish the conditions under which the semiflow $\mathfrak{F}$, determined by the parameters $n=2$, $F(s,u)=\psi(s)K$ is unstable in terms of the measures $(S[u],S[u])$.

 For  functional $W_0[u]=S[u]$ the following equality holds
$$
\frac{dW_0[\mathfrak{F}^t(u)]}{dt}\Big|_{t=0}=\tr A\varphi(W_0[u])W_0[u]+2\psi(W_0[u])S[u,K].
$$
 We get from   Brunn--Minkowskij inequality $S[u,K]\ge\sqrt{\pi}\sqrt{S[u]}$  the estimate
$$
\frac{dW_0[\mathfrak{F}^t(u)]}{dt}\Big|_{t=0}\ge\tr A\varphi(W_0[u])W_0[u]+2\sqrt{\pi}\psi(W_0[u])\sqrt{W_0[u]}.
$$
The  equation of comparison is of the form
$$
\frac{d\xi_0}{dt}=\tr A\xi_0\varphi(\xi_0)+2\sqrt{\pi}\psi(\xi_0)\sqrt{\xi_0}.
$$
 We obtain from  Martynyuk–-Obolenskij criterion the conditions of instability of the solution $\xi_0=0$ for equation of comparison
$$
\lim\inf\limits_{s\to 0+}\sqrt{\frac{\pi}{s}}\frac{\psi(s)}{\varphi(s)}>-\frac{\tr A}{2}.
$$
This inequality guarantees the instability  in terms of  measures $(S[u],S[u])$ of considered semiflow.

\section{Conclusion}
The obtained results significantly extend  a region of applicability of comparison method in the stability theory and generalize some of the results of  general theory for set differential equations. In particular,  these results allow to overcome known difficulties associated with the notion of asymptotic stability of solutions for these classes of equations \cite{lak}.
On the other hand, the considered semiflows are  natural generalization for families of sets of attainability for  linear control systems. The use of classic geometric inequalities  and results of convex geometry ascending to works  of G. Minkowskij and A.D. Aleksandrov, allows us to get the conditions of stability and practical stability  for semiflows in terms of  measures having a specific geometric meaning. For further study it is of interest to dissemination the obtained results to other classes of semiflows in  space $\mathrm{conv}\,(\Bbb R^n)$.

\end{document}